\documentclass[leqno,11pt]{amsart}
\usepackage{amsmath,amssymb}
\usepackage{color}
\usepackage{comment}

\theoremstyle{plain}
\newtheorem{theorem}{Theorem}[section]
\newtheorem{lemma}[theorem]{Lemma}
\newtheorem{prop}[theorem]{Proposition}
\newtheorem{corollary}[theorem]{Corollary}
\theoremstyle{definition}

\newtheorem{remark}[theorem]{Remark}

\def\rn{\mathbb R\sp n}
\def\R{\mathbb R}
\def\N{\mathbb N}
\def\hra{\to}
\def\M{\mathcal M}
\def\dist{\operatorname{dist}}
\def\limsup{\operatornamewithlimits{lim\,sup}}

\def\ve{\varepsilon}

\def\VV{\mathcal V}

\newtoks\by
\newtoks\paper
\newtoks\book
\newtoks\jour
\newtoks\yr
\newtoks\pages
\newtoks\vol
\newtoks\publ
\newtoks\eds
\newtoks\proc
\def\ota{{\hbox{???}}}
\def\cLear{\by=\ota\paper=\ota\book=\ota\jour=\ota\yr=\ota
\pages=\ota\vol=\ota\publ=\ota}
\def\endpaper{\the\by, \textit{\the\paper},
{\the\jour} \textbf{\the\vol} (\the\yr), \the\pages.\cLear}
\def\endbook{\the\by, \textit{\the\book}, \the\publ.\cLear}
\def\endprep{\the\by, \textit{\the\paper}, \the\jour.\cLear}
\def\endproc{\the\by, \textit{\the\paper}, \the\publ, \the\pages.\cLear}
\def\name#1#2{#1 #2}
\def\et{ and }

\numberwithin{equation}{section}

\begin{document}

\title[Banach algebras of weakly differentiable functions]
{Banach algebras of weakly differentiable functions}

\author {Andrea Cianchi, Lubo\v s Pick and Lenka Slav\'ikov\'a}
\address{Dipartimento di Matematica e Informatica \lq\lq U. Dini"\\
Universit\`a di Firenze\\
Piazza Ghiberti 27\\
50122 Firenze\\
Italy} \email{cianchi@unifi.it}

\address{Department of Mathematical Analysis\\
Faculty of Mathematics and Physics\\
Charles University\\
Sokolovsk\'a~83\\
186~75 Praha~8\\
Czech Republic} \email{pick@karlin.mff.cuni.cz}

\address{Department of Mathematical Analysis\\
Faculty of Mathematics and Physics\\
Charles University\\
Sokolovsk\'a~83\\
186~75 Praha~8\\
Czech Republic} \email{slavikova@karlin.mff.cuni.cz}


\thanks{This research was partly supported by the Research Project of Italian Ministry of University and
Research (MIUR) ``Elliptic and parabolic partial differential
equations: geometric aspects, related inequalities, and
applications" 2012,  by  GNAMPA  of Italian INdAM (National
Institute of High Mathematics), and by the grant P201-13-14743S of
the Grant Agency of the Czech Republic.}

\begin{abstract}
{The question is addressed of when a Sobolev type space, built upon
a general rearrangement-invariant norm, on an $n$-dimensional
domain, is a Banach algebra under pointwise multiplication of
functions. A sharp balance condition among the order of the Sobolev
space, the strength of the norm, and the (ir)regularity of the
domain is provided for the relevant Sobolev space to be a Banach
algebra. The regularity of the domain is described in terms of its
isoperimetric function. Related results on the boundedness of the
multiplication operator into lower-order Sobolev type spaces are
also established. The special cases of Orlicz-Sobolev and
Lorentz-Sobolev spaces are discussed in detail. New results  for
classical Sobolev spaces on possibly irregular domains follow as
well.}
\end{abstract}

\maketitle

\section{Introduction and  main results}\label{S:intro}
{The Sobolev space $W^{m,p}(\Omega)$ of those functions in an open
set $\Omega \subset \rn$, $n\geq 2$, whose weak derivatives up to
the order $m$ belong to $L^p(\Omega)$, is classically well known to
be a Banach space for every $m \in \N$ and $p \in [1, \infty]$. In
particular, the sum of any two functions from $W^{m,p}(\Omega)$
always still belongs to $W^{m,p}(\Omega)$. The situation is quite
different if the operation of  sum is replaced by product. In fact,
membership of functions to a Sobolev space need not be preserved
under multiplication. Hence, $W^{m,p}(\Omega)$ is not a Banach
algebra in general. A standard result in the theory of Sobolev
spaces tells us that if $\Omega$ is regular, say a bounded domain
with the cone property,  then $W^{m,p}(\Omega)$ is indeed a Banach
algebra if and only if either $p>1$ and $pm>n$, or $p =1$ and $m
\geq n$. Recall that this amounts to the existence of a constant $C$
such that
\begin{equation}\label{classical}
\|uv\|_{W\sp mX(\Omega)} \leq C \|u\|_{W\sp mX(\Omega)}\|v\|_{W\sp
mX(\Omega)}
\end{equation}
for every $u, v\in W\sp mX(\Omega)$. We refer to Section 6.1  of the
monograph \cite{MS}  for this result, where a comprehensive updated
treatment of properties of Sobolev functions under product can be
found. See also \cite[Theorem~5.23]{RAA} for a proof of the
sufficiency part of the result.}
\par
{In the present paper abandon this
classical setting, and address the question of the validity of an
inequality of the form \eqref{classical}  in a much more general
framework. } Assume  that $\Omega$ is just a domain in $\rn$, namely
an~open connected set, with finite
 Lebesgue measure $|\Omega|$, which, without loss of
generality, will be assumed to be equal to $1$. Moreover, suppose
that $L^p(\Omega)$ is replaced with an arbitrary
rearrangement-invariant space $X(\Omega)$,  loosely speaking, a
Banach space of measurable functions endowed with a norm depending
only on the measure of level sets of functions. We refer to the next
section for precise definitions concerning function spaces. Let us
just recall here that, besides Lebesgue spaces, Lorentz and Orlicz
spaces are classical instances of rearrangement-invariant spaces.
\par
Given any $m \in \N$ and any rearrangement-invariant space $X(\Omega)$, consider
the $m$-th order Sobolev type space $\VV ^m X(\Omega)$  built upon
$X(\Omega)$, and defined as the collection of all $m$ times weakly
differentiable functions $u:\Omega\to \R$ such that $|\nabla\sp
mu|\in X(\Omega)$. Here,
$\nabla ^m u$ denotes the vector of all $m$-th order weak
derivatives of $u$, and $|\nabla\sp mu|$ stands for its length.
For notational convenience, we also set $\nabla\sp0 u=u$ and $\VV^0X(\Omega)=X(\Omega)$.
Given any fixed ball $B\subset \Omega$, we define the functional
 $\|\cdot\|_{\VV\sp m X(\Omega)}$ by
\begin{equation}\label{E:norm-in-vdot}
\|u\|_{\VV\sp m X(\Omega)}=\sum _{k=0}^{m-1}\|\nabla ^k u\|_{L^1(B)}
+ \|\nabla ^m u\|_{X(\Omega)}
\end{equation}
for $u\in \VV\sp mX(\Omega)$. Observe that in the definition
of $\VV^m X(\Omega)$ it is only required that the derivatives of the
highest order $m$ of $u$ belong to $X(\Omega )$. This assumption
does not ensure, for an arbitrary domain $\Omega$, that also $u$ and
its derivatives up to the order $m-1$ belong to $X(\Omega )$, or
even to $L^1(\Omega )$. However, owing to a standard Poincar\'e
inequality, if  $u \in \VV^m X(\Omega)$, then $|\nabla ^k u|\in
L^1(B)$ for  $k =0, \dots ,m-1$, for every ball $B \subset \Omega$. It follows that the functional $\|\cdot\|_{\VV\sp m X(\Omega)}$ is a norm on $\VV\sp mX(\Omega)$. Furthermore, a~standard argument shows that $\VV\sp mX(\Omega)$ is a Banach space
equipped with this norm, which results in equivalent norms under
replacements of $B$ with other balls.
\par
We shall exhibit minimal conditions on $m$, $\Omega$ and $\| \cdot
\|_{X(\Omega)}$ for $\VV\sp mX(\Omega)$ to be a~Banach algebra under
pointwise multiplication of functions, namely for an inequality of
the form
\begin{equation*}
\|uv\|_{\VV\sp mX(\Omega)} \leq C \|u\|_{\VV\sp
mX(\Omega)}\|v\|_{\VV\sp mX(\Omega)}
\end{equation*}
to hold for some constant $C$ and every $u, v\in \VV\sp mX(\Omega)$.
Variants of this inequality, where $\VV\sp mX(\Omega)$ is replaced
by a lower-order Sobolev space on the left-hand side, are also dealt
with.
\par
In our discussion,  we neither a priori assume any regularity on
$\Omega$, nor we assume that $X(\Omega)$ is a Lebesgue space (or any
other specific space).  We shall exhibit a balance condition between
the degree of regularity of $\Omega$, the order of differentiation
$m$, and the strength of the norm in $X(\Omega)$ ensuring that
$\VV\sp mX(\Omega)$ be a Banach algebra. The dependence on
$X(\Omega)$ is only through the representation norm $\| \cdot
\|_{X(0,1)}$ of $\| \cdot \|_{X(\Omega)}$. In particular, the
associate norm $\| \cdot \|_{X'(0,1)}$ of $\| \cdot \|_{X(0,1)}$, a
kind of measure theoretic dual norm of $\| \cdot \|_{X(0,1)}$, will
be relevant. \\ {As for our
assumptions on the domain $\Omega$, a key role in their formulation
will be played by the relative isoperimetric inequality. Let us
recall that the discovery of the link between isoperimetric
inequalities
 and Sobolev type inequalities can be traced back to the work of Maz'ya on one
 hand (\cite{Ma1960, Ma1961}), who proved the equivalence of general Sobolev inequalities to either isoperimetric or isocapacitary inequalities,
  and that of Federer and Fleming on the other hand (\cite{FF}) who used the standard isoperimetric inequality by De Giorgi (\cite{DeG})
   to exhibit the best constant in  the Sobolev inequality for $W^{1,1}(\rn)$. The detection of optimal constants in classical Sobolev
   inequalities continued in the contributions \cite{Moser}, \cite{Ta}, \cite{Aubin}, where crucial use of De Giorgi's isoperimetric inequality
   was again made. An~extensive research followed, along diverse directions, on the interplay between isoperimetric and Sobolev
   inequalities. We  just mention  the papers  \cite{AFT,
BCR1, BWW, BH, BLbis, BK, BK1, CK, Cheeger, Ci_ind, Ci1, CFMP1,
CP_gauss, EKP, EFKNT, GG, Gr, HK, HS1, HS2, KM, Kl, Ko, Lio, LPT,
LYZ, M, Swa, Zh} and the monographs \cite{BZ, CDPT, chavel, GG-book,
Heb, Mabook, Saloff}.}
\par Before stating our most general result, let us focus on the
situation when $m$ and $X(\Omega)$ are arbitrary, but $\Omega$ is
still, in a sense, a best possible domain. This is the case when
$\Omega$ is a John domain. Recall that a~bounded open set $\Omega$
in $\rn$ is called a~\textit{John domain} if there exist a constant
$c \in (0,1)$, an $l\in(0,\infty)$ and a point $x_0 \in \Omega$ such that for every $x
\in \Omega$ there exists a rectifiable curve $\varpi : [0, l] \to
\Omega$, parameterized by arclength, such that $\varpi (0)=x$,
$\varpi (l) = x_0$, and
$$
{\dist}\, (\varpi (r) , \partial \Omega ) \geq c r \qquad \hbox{for
$r \in [0, l]$.}
$$
Lipschitz domains, and domains with the cone property are customary
instances of John domains.
\par
When $\Omega$ is any John domain, a necessary and sufficient
condition for  $\VV\sp mX(\Omega)$ to be a Banach algebra  is
provided by the following result.

\begin{theorem}\label{T:john}
Let $m,n\in\N$, $n\geq 2$. Assume that $\Omega$ is a~John
domain in $\R\sp n$. Let $\|\cdot\|_{X(0,1)}$ be a
rearrangement-invariant function norm. Then $\VV\sp{m}X(\Omega)$ is
a~Banach algebra if and only if
\begin{equation}\label{E:john-condition}
\|r\sp{-1+\frac mn}\|_{X'(0,1)}<\infty.
\end{equation}
\end{theorem}

{As a consequence of
Theorem~\ref{T:john}, and of the characterization of Sobolev
embeddings into $L^\infty(\Omega)$, we have the following
corollary.}

{\begin{corollary}\label{C:embedding-to-linfty-and-algebra-john}
Let $m,n\in\N$, $n\geq 2$. Assume that $\Omega$ is a~John domain in
$\R\sp n$. Let $\|\cdot\|_{X(0,1)}$ be a rearrangement-invariant
function norm. Then the Sobolev space $\VV\sp mX(\Omega)$ is
a~Banach algebra if and only if $\VV\sp mX(\Omega)\to
L\sp{\infty}(\Omega)$.
\end{corollary}}

Let us now turn to the general case. Regularity on $\Omega$ will be
imposed in terms of its isoperimetric function $I_{\Omega} : [0, 1]
\to [0, \infty ]$, introduced in \cite{Ma1960}, and
given by
\begin{equation}\label{isopomega}
I_{\Omega} (s) = \inf \left\{P(E,\Omega) : E \subset \Omega, s \leq
|E| \leq \tfrac 12 \right\} \quad \textup{if}\ s \in [0,\tfrac 12],
\end{equation}
and $I_{\Omega} (s) = I_{\Omega} (1 -s)$ if $s \in (\frac 12, 1]$.
Here, $P(E,\Omega)$ denotes the perimeter of a measurable set $E$
relative to $\Omega$, which agrees with $\mathcal H^{n-1}(\Omega
\cap
\partial ^M E)$,
where $\partial ^ME$ denotes the essential boundary of $E$, in the
sense of geometric measure theory, and $\mathcal H^{n-1}$ stands for
$(n-1)$-dimensional Hausdorff measure. The very definition of
$I_{\Omega}$ implies the relative isoperimetric inequality in
$\Omega$, which tells us that
\begin{equation}\label{isop2}
P(E,\Omega) \geq I_{\Omega}(|E|),
\end{equation}
for every measurable set $E \subset \Omega$. In other words, $
I_{\Omega}$ is the largest non-decreasing function in $[0,
\tfrac{1}{2}]$, symmetric about $\tfrac 12$, which renders
\eqref{isop2} true.
\par
The degree of regularity of $\Omega$ can be described in terms of
the rate of decay  of $I_\Omega (s)$ to $0$ as $s \to 0$.
Heuristically speaking, the faster $I_\Omega $ decays to $0$, the
less regular  $\Omega$ is. For instance, the isoperimetric function
$I_\Omega$ of any John domain $\Omega\subset\rn$ is known to satisfy
\begin{equation}\label{Ijohn}
I_\Omega (s) \approx s^{\frac 1{n'}}
\end{equation}
near $0$, where  $n'=\frac n{n-1}$. Here, and in what follows, the
notation $f \approx g$ mans that the real-valued functions $f$ and
$g$ are equivalent, in the sense that there exist positive constants
$c,C$ such that $cf(c \cdot )\leq g(\cdot )\leq Cf(C \cdot)$. Notice
that \eqref{Ijohn} is the best (i.e. slowest) possible decay of
$I_\Omega$, since, if  $\Omega$ is any domain, then
\begin{equation}\label{Igeneral}
\frac{I_\Omega (s)}{s^{\frac 1{n'}}} \leq C \quad \hbox{ for $s \in
(0, 1]$,}
\end{equation}
for some constant $C$  \cite[Proposition~4.1]{CPS}.
\par
What enters in our characterization of Sobolev algebras is, in fact,
just a lower bound for $I_\Omega $. We shall thus work with classes
of domains whose isoperimetric function admits  a lower bound  in
terms of some {non-decreasing function $I: (0, 1)\to (0, \infty)$.
The function $I$ will be continued by continuity at $0$  when
needed.}
 Given any such function $I$, we denote by $\mathcal J_I$
the collection of all domains $\Omega\subset \mathbb R\sp n$ such
that
\begin{equation}\label{isop-ineq}
I_{\Omega} (s) \geq c I(c s) \quad\textup{for }\ s\in(0,\tfrac12],
\end{equation}
for some constant $c>0$. {The assumption that $I(t)>0$ for
$t\in(0,1)$ is consistent with the fact that $I_{\Omega}(t)>0$ for
 $t\in(0,1)$,  owing to the connectedness of $\Omega$   \cite[Lemma
5.2.4]{Mabook}.}
\\ In particular, if $I(s) = s^\alpha$ for $s \in
(0,1)$, for some
  $\alpha \in
[\frac 1{n'},\infty)$,  we denote $\mathcal J_I$ simply by $\mathcal
J_\alpha$, and call it a~\textit{Maz'ya class}. Thus, a domain
$\Omega \in \mathcal J_\alpha$ if there exists a positive constant
$C$ such that
\[
I_\Omega (s) \geq C s^\alpha \quad \hbox{for every $s
\in (0, \frac 12]$.}
\]
Observe that, thanks to \eqref{Ijohn}, any John domain belongs to
the class $\mathcal J_{\frac 1{n'}}$.
\par
{Our most general result about Banach
algebras of Sobolev spaces is stated in the next theorem. Let us
emphasize that, if $\Omega$ is not a regular domain, it brings new
information even in the standard case when $X(\Omega) =
L^p(\Omega)$.}

\begin{theorem}\label{T:algebra-for-W}
Assume that $m,n\in\N$, $n\geq 2$, and let $\|\cdot\|_{X(0,1)}$ be
a~rearrangement-invariant~function norm. Assume that $I$ is a
positive non-decreasing function on $(0,1)$. If $\Omega \in \mathcal
J_I$ and
\begin{equation}\label{E:psi-in-associate}
\left\|\frac{1}{I(t)}\left(\int_0\sp
t\frac{ds}{I(s)}\right)\sp{m-1}\right\|_{X'(0,1)}<\infty,
\end{equation}
then
\begin{equation*}
\VV\sp mX(\Omega) \textup{\quad is a~Banach algebra},
\end{equation*}
{or, equivalently, there exists a constant $C$ such that
\begin{equation}\label{algebra}
\|uv\|_{\VV\sp mX(\Omega)} \leq C \|u\|_{\VV\sp
mX(\Omega)}\|v\|_{\VV\sp mX(\Omega)}
\end{equation}
for every $u, v\in \VV\sp mX(\Omega)$.}
\\
Conversely, if, in addition,
\begin{equation}\label{E:I-divided-by-power-increasing}
\frac{I(t)}{t\sp{\frac1{n'}}}\quad\textit{is equivalent to
a~non-decreasing function on}\ (0,1),
\end{equation}
then~\eqref{E:psi-in-associate} is sharp, in the sense that if
$\VV\sp mX(\Omega)$ is a~Banach algebra for every $\Omega \in
\mathcal J_I$, then~\eqref{E:psi-in-associate} holds.
\end{theorem}

\begin{remark}\label{R:best-possible-behavior}
Assumption \eqref{E:I-divided-by-power-increasing} is not
restrictive in view of \eqref{Igeneral}, and can just be regarded as
 a qualification of the latter.
\end{remark}

\begin{remark}\label{conv-int-I}
If condition \eqref{E:psi-in-associate} holds for some $X(0, 1)$ and
$m$, then necessarily
$$\int _0\frac{ds}{I(s)} < \infty.$$ This
is obvious for $m\geq 2$, whereas  it follows from~\eqref{l1linf}
below for $m=1$.
\end{remark}

An analogue of Corollary
\ref{C:embedding-to-linfty-and-algebra-john} is provided by the
following statement.

\begin{corollary}\label{C:embedding-to-linfty-and-algebra}
Let $m,n\in\N$, $n\geq 2$, and let $\|\cdot\|_{X(0,1)}$ be
a~rearrangement-invariant~function norm.  Assume that $I$ is a
positive non-decreasing function on $(0,1)$
satisfying~\eqref{E:I-divided-by-power-increasing}. Then the Sobolev
space $\VV\sp mX(\Omega)$ is  a~Banach algebra for all domains
$\Omega$ satisfying~\eqref{isop-ineq} if and only if $\VV\sp
mX(\Omega)\to L\sp{\infty}(\Omega)$ for every $\Omega\in\mathcal
J_I$.
\end{corollary}

The next corollary of Theorem \ref{T:algebra-for-W}  tells us that
$\VV\sp mX(\Omega)$ is always a Banach algebra, whatever $X(\Omega)$
is, provided that $I_\Omega$ is sufficiently well behaved near $0$,
depending on $m$.

\begin{corollary}\label{C:linfty}
Let $m,n\in\N$, $n\geq 2$, and let  $\Omega \in \mathcal J_I$ for
some positive  non-decreasing function $I$ on $(0,1)$. Suppose
that
\begin{equation}\label{linfty1}
\limsup_{t\to0_+}\frac{1}{I(t)}\left(\int_0\sp
t\frac{dr}{I(r)}\right)\sp{m-1}<\infty.
\end{equation}
Then the Sobolev space $\VV\sp mX(\Omega)$ is a~Banach algebra for every rearrangement-invariant space
$X(\Omega)$.
\end{corollary}

\begin{remark}\label{VVW}
{\rm Besides  $\VV^m X(\Omega)$, one can consider the $m$-th order
Sobolev type space $V^m X(\Omega)$ of those functions $u$ such that
the norm
\begin{equation}\label{E:norm-in-vmx}
\|u\|_{V^m X(\Omega)} = \sum _{k=0}^{m-1}\|\nabla ^k
u\|_{L^1(\Omega)} + \|\nabla ^m u\|_{X(\Omega)}
\end{equation}
is finite, and the space  $W^m X(\Omega)$ of those functions whose
norm
$$
\|u\|_{W^m X(\Omega)} = \sum _{k=0}^{m} \|\nabla ^k u\|_{X(\Omega)}
$$
is finite. Both $V^m X(\Omega)$ and $W^m X(\Omega)$ are Banach
spaces. Since any rearrangement-invariant space is embedded into $L^1(\Omega)$, one has that
\begin{equation}\label{inclusions}
W^m X(\Omega) \to V^m X(\Omega) \to \VV^m X(\Omega),
\end{equation}
and the inclusions are strict, as noticed above, unless $\Omega$
satisfies some additional regularity assumption. In particular, if
\begin{equation}\label{E:2}
\inf_{t\in(0,1)}\frac{I_{\Omega}(t)}{t}>0,
\end{equation}
then
\begin{equation}\label{inclusions1}
V^m X(\Omega) = \VV^m X(\Omega),
\end{equation}
with equivalent norms, as a consequence of \cite[Theorem
5.2.3]{Mabook} and of the closed graph theorem.
\\ If assumption \eqref{E:2} is strengthened to
\begin{equation}\label{E:convergence-of-I-omega}
\int _0  \frac {ds}{I_{\Omega}(s)} < \infty,
\end{equation}
then, in fact,
\begin{equation}\label{inclusions2}
W^m X(\Omega) = V^m X(\Omega) = \VV^m X(\Omega),
\end{equation}
up to equivalent norms \cite[Proposition~4.5]{CPS}.
\\ If $\Omega \in \mathcal J_I$ for some $I$ fulfilling
\eqref{E:psi-in-associate}, then, by Remark \ref{conv-int-I},
condition \eqref{E:convergence-of-I-omega} is certainly satisfied.
 Thus, by the first part of
Theorem \ref{T:algebra-for-W}, assumption \eqref{E:psi-in-associate}
is sufficient also for $V^m X(\Omega)$ and $W^m X(\Omega)$ to be
Banach algebras. Under \eqref{E:I-divided-by-power-increasing}, such
assumption is also necessary, provided that the function $I$ is a
priori known to fulfil either
\begin{equation}\label{E:2I}
\inf_{t\in(0,1)}\frac{I(t)}{t}>0,
\end{equation}
or
\begin{equation}\label{E:convergence-of-I-omegaI}
\int _0  \frac {ds}{I(s)} < \infty,
\end{equation}
according to whether $V^m X(\Omega)$ or $W^m X(\Omega)$ is in
question. Let us emphasize that these additional assumptions cannot
be dispensed with in general. For instance, it is easily seen that $W^m
L^\infty(\Omega)$ is always a Banach algebra, whatever $m$ and
$\Omega$ are.}

\end{remark}

\medskip

{So far we have analyzed the question
of whether $\VV^mX(\Omega)$ is a Banach algebra, namely of the
validity of inequality \eqref{algebra}. We now focus on inequalities
in the spirit of \eqref{algebra}, where the space $\VV^mX(\Omega)$
is replaced with a lower-order Sobolev space $\VV^{m-k}X(\Omega)$ on
the left-hand side.}
The statement of our result in this connection requires the notion
of the fundamental function $\varphi_X : [0,1] \to [0, \infty)$ of a
rearrangement-invariant
 function norm $\|\cdot\|_{X(0,1)}$. Recall that
\begin{equation}\label{fund}
\varphi_X(t)=\|\chi_{(0,t)}\|_{X(0,1)} \quad \hbox{for $t \in (0,
1]$.}
\end{equation}

\begin{theorem}\label{T:algebra-restricted}
Let $m$, $n$, $k\in \mathbb N$, $n\geq 2$, $1 \leq k \leq m$, and
let $\|\cdot\|_{X(0,1)}$ be a~rearrangement-invariant function norm. Assume
that $I$ is a positive non-decreasing function on $(0,1)$. If
$\Omega \in \mathcal J_I$ and
\begin{equation}\label{E:power-m+k}
\sup _{t \in (0,1)}\frac 1{\varphi _X(t)}\left(\int_0^t
\frac{\,ds}{I(s)}\right)^{m+k}< \infty,
\end{equation}
then
{\begin{equation}\label{E:P_bounded}
\|uv\|_{\VV^{m-k}X(\Omega)} \leq C \|u\|_{\VV\sp
mX(\Omega)}\|v\|_{\VV\sp mX(\Omega)}
\end{equation}
for every $u, v\in \VV\sp mX(\Omega)$.}

Conversely, if, in addition, \eqref{E:I-divided-by-power-increasing}
is fulfilled, then~\eqref{E:power-m+k} is sharp, in the sense that
if~\eqref{E:P_bounded} is satisfied for all domains $\Omega \in
\mathcal J_I$, then~\eqref{E:power-m+k} holds.
\end{theorem}

\begin{remark}\label{VVWrestrict}
Considerations on the replacement of the space $\VV\sp mX(\Omega)$
with either the space $V\sp mX(\Omega)$, or $W\sp mX(\Omega)$ in
Theorem \ref{T:algebra-restricted} can be made, which are analogous
to those of Remark \ref{VVW} about Theorem \ref{T:algebra-for-W}.
\end{remark}

{In the borderline case when $k=0$,
condition \eqref{E:power-m+k} is (essentially) weaker than
\eqref{E:psi-in-associate} - see Proposition
\ref{R:psi-implies-fundamental} and Remark \ref{stronger}, Section
\ref{S:pointwise-inequalities}. The next result asserts that it
``almost" implies inequality \eqref{algebra}, in that it yields
an inequality of that form, with the borderline terms $u\nabla ^mv$
and $v \nabla ^m u$ missing in the Leibniz formula for the $m$-th
order derivative of the product $uv$.}

{\begin{theorem}\label{C:gradients} Let
$m,n\in\N$, $m,n\geq 2$, and let $\|\cdot\|_{X(0,1)}$ be
a~rearrangement-invariant function norm. Assume that $I$ is a
positive non-decreasing function on $(0,1)$. If $\Omega \in \mathcal
J_I$ and
\begin{equation}\label{E:fundamental-estimate-first}
\sup_{t\in(0,1)}\frac1{\varphi_X(t)}\left(\int_0\sp
t\frac{ds}{I(s)}\right)\sp m<\infty\,,
\end{equation}
then there exists a~positive constant $C$ such that
\begin{equation}\label{E:corollary-lower}
\sum_{k=1}\sp{m-1}\| |\nabla\sp ku| |\nabla\sp{m-k}v|
\|_{X(\Omega)}\leq C\|u\|_{\VV\sp mX(\Omega)}\|v\|_{\VV\sp
mX(\Omega)}
\end{equation}
for every $u,v\in\VV\sp mX(\Omega)$.
\\ {Conversely, if, in addition,
\eqref{E:I-divided-by-power-increasing} is fulfilled, then~
\eqref{E:fundamental-estimate-first} is sharp, in the sense that
if~\eqref{E:corollary-lower} is satisfied for all domains $\Omega
\in \mathcal J_I$, then~\eqref{E:fundamental-estimate-first} holds.}
\end{theorem}}

Theorem \ref{T:algebra-for-W} enables us, for instance, to
characterize
 the
Lorentz--Zygmund-Sobolev spaces and Orlicz-Sobolev spaces which are
Banach algebras for all domains $\Omega\in \mathcal J_{\alpha}$.
\\ Let us first focus on the Lorentz--Zygmund-Sobolev spaces
$\VV\sp{m}L\sp{p,q;\beta}(\Omega)$. Recall (see e.g.~\cite[Theorem~9.10.4]{fs-1} or~\cite[Theorem~7.4]{glz}) that a necessary and
sufficient condition for $L\sp{p,q;\beta}(\Omega)$ to be a
rearrangement-invariant space is that the parameters $p, q, \beta$
satisfy either of the following conditions:
\begin{equation}\label{E:lz_bfs}
\begin{cases}
1<p<\infty,\ 1\leq q\leq\infty,\ \beta\in\R;\\
p=1,\ q=1,\ \beta\geq0;\\
p=\infty,\ q=\infty,\ \beta\leq 0;\\
p=\infty,\ 1\leq q<\infty,\ \beta+\frac1q<0.
\end{cases}
\end{equation}

\begin{prop}\label{T:lz}
Let $m,n\in\N$, $n\geq 2$, and let $\alpha\in[\frac1{n'},\infty)$.
Assume that $1\leq p,q\le\infty$, $\beta\in\R$ and one of the
conditions in~\eqref{E:lz_bfs} is in force. Then $\VV\sp
mL\sp{p,q;\beta}(\Omega)$ is a~Banach algebra for every domain
$\Omega\in\mathcal J_{\alpha}$ if and only if $\alpha<1$ and either
of the following conditions is satisfied:
\begin{equation}\label{E:lz-algebra}
\begin{cases}
m(1-\alpha)>\frac 1p,\\
m(1-\alpha)=\frac 1p,\ q=1,\ \beta\geq0,\\
m(1-\alpha)=\frac 1p,\ q>1,\ \beta>\frac{1}{q'}.
\end{cases}
\end{equation}
\end{prop}

The Orlicz-Sobolev spaces $\VV\sp{m}L\sp{A}(\Omega)$  are the object
of the next result. In particular, it recovers a result from
\cite{Ci}, dealing with the case of regular domains.

\begin{prop}\label{T:orlicz}
Let $m\in\N$, $n\geq 2$, and $\alpha\in[\frac1{n'},\infty)$. Let $A$
be a~Young function. Then $\VV\sp mL\sp{A}(\Omega)$ is a~Banach
algebra for every domain $\Omega\in\mathcal J_{\alpha}$ if and only
if  $\alpha<1$ and either of the following conditions is satisfied:
\begin{equation}\label{E:orlicz-algebra}
\begin{cases}
m \geq \tfrac 1{1-\alpha}, \\
m < \tfrac 1{1-\alpha} \quad \hbox{and}\quad   \int ^\infty
\big(\frac t{A(t)}\big)^{\frac {(1-\alpha)m}{1- (1-\alpha )m}}dt <
\infty.
\end{cases}
\end{equation}
\end{prop}

The Lorentz--Zygmund-Sobolev and Orlicz-Sobolev spaces for which the
product operator is bounded into a lower-order space  for every
$\Omega\in \mathcal J_{\alpha}$ can be characterized via Theorem
\ref{T:algebra-restricted}.

\begin{prop}\label{T:lz-reduced}
Let $n,m,k\in\N$, $n\geq 2$, $1 \leq k \leq m$, and let
$\alpha\in[\frac1{n'},\infty)$. Assume that $1\leq p,q\le\infty$,
$\beta\in\R$ and one of the conditions in~\eqref{E:lz_bfs} is in
force. {Then, for every domain
$\Omega\in\mathcal J_{\alpha}$ there exists a constant $C$ such that
\begin{equation*}
\|uv\|_{\VV^{m-k}L\sp{p,q;\beta}(\Omega)} \leq C \|u\|_{\VV\sp
mL\sp{p,q;\beta}(\Omega)}\|v\|_{\VV\sp m L\sp{p,q;\beta}(\Omega)}
\end{equation*}
for every $u, v \in \VV\sp m L\sp{p,q;\beta}(\Omega)$}
 if and only if
$\alpha<1$, and either of the following conditions is satisfied:
\begin{equation}\label{E:lz-reduced}
\begin{cases}
(m+k)(1-\alpha)>\frac 1p,\\
(m+k)(1-\alpha)=\frac 1p,\ \beta\geq0.
\end{cases}
\end{equation}
\end{prop}

\begin{prop}\label{T:orlicz-reduced}
Let $n,m,k\in\N$, $n\geq 2$, $1 \leq k \leq m$, and let
$\alpha\in[\frac1{n'},\infty)$. Let $A$ be a~Young function.
{Then,
 for every domain $\Omega\in\mathcal J_{\alpha}$ there exists a
constant $C$ such that
\begin{equation*}
\|uv\|_{\VV^{m-k}L\sp{A}(\Omega)} \leq C \|u\|_{\VV\sp
mL\sp{A}(\Omega)}\|v\|_{\VV\sp m L\sp{A}(\Omega)}
\end{equation*}
for every $u, v \in \VV\sp m L\sp{p,q;\beta}(\Omega)$} if and only
if $\alpha<1$ and either of the following conditions is satisfied:
\begin{equation}\label{E:orlicz-reduced}
\begin{cases}
(m+k) \geq \tfrac 1{1-\alpha}, \\
(m+k) < \tfrac 1{1-\alpha} \quad \hbox{and}\quad  A(t) \geq C
t^{\frac 1{(1-\alpha)(m+k)}} \quad \hbox{for large } \,\, t\,,
\end{cases}
\end{equation}
for some positive constant $C$.
\end{prop}

\section{Background}\label{S:background}
We denote by $\M(\Omega)$
the set of all Lebesgue measurable functions from ~$\Omega$ into
$[-\infty , \infty]$. We also define $\M_+(\Omega)= \{u \in
\M(\Omega): u \geq 0\}$, and $\M _0(\Omega)= \{u \in \M(\Omega): u\
\textup{is finite a.e. in}\ \Omega\}$.

The \textit{decreasing rearrangement} $u\sp* : (0, 1) \to [0, \infty
]$ of a function $u \in \M(\Omega)$  is defined as
$$
u\sp*(s)=\sup\{t\in\R:\,|\{x\in \Omega :\,|u(x)|>t\}|>s\}  \quad
\hbox{for $ s\in(0,1)$.}
$$
We also define $u^{**} : (0,1) \to [0, \infty]$ as
$$
u^{**}(s) = \frac{1}{s}\int _0^s u^*(r)\, dr.
$$

We say that a functional $\|\cdot\|_{X(0,1)}: \M _+ (0, 1) \to
[0,\infty]$  is a \textit{function norm}, if, for all $f$, $g$ and
$\{f_j\}_{j\in\N}$ in $\M_+(0,1)$, and every $\lambda \geq0$, the
following properties hold:
\begin{itemize}
\item[(P1)]\qquad $\|f\|_{X(0,1)}=0$ if and only if $f=0$ a.e.;
$\|\lambda f\|_{X(0,1)}= \lambda\|f\|_{X(0,1)}$; \par\noindent
\qquad $\|f+g\|_{X(0,1)}\leq \|f\|_{X(0,1)}+ \|g\|_{X(0,1)}$;
\item[(P2)]\qquad $  f \le g$ a.e.\  implies $\|f\|_{X(0,1)} \le
\|g\|_{X(0,1)}$; \item[(P3)]\qquad $  f_j \nearrow f$ a.e.\ implies
$\|f_j\|_{X(0,1)} \nearrow \|f\|_{X(0,1)}$;
\item[(P4)]\qquad $\|1\|_{X(0,1)}<\infty$; \item[(P5)]\qquad  $\int_{0}\sp1 f(x)\,dx \le C
\|f\|_{X(0,1)}$ for some constant $C$ independent of $f$.
\end{itemize}
If, in addition,
\begin{itemize}
\item[(P6)]\qquad $\|f\|_{X(0,1)} = \|g\|_{X(0,1)}$ whenever
$f\sp* = g\sp *$,
\end{itemize}
we say that $\|\cdot\|_{X(0,1)}$ is a
\textit{rearrangement-invariant function norm}.

Given a rearrangement-invariant function norm
$\|\cdot\|_{X(0,1)}$, the space $X(\Omega)$ is
defined as the collection of all  functions  $u \in\M(\Omega)$ such that the expression
\[
\|u\|_{X(\Omega)}=\|u\sp* \|_{X(0,1)}
\]
is finite. Such expression defines a norm on $X(\Omega)$, and
the latter  is a Banach space  endowed with this norm, called a~\textit{rearrangement-invariant space}. Moreover,
$X(\Omega) \subset \M_0(\Omega)$ for any rearrangement-invariant space $X(\Omega)$.

With any rearrangement-invariant~function norm $\|\cdot\|_{X(0,1)}$,
it is associated another functional on $\M_+(0,1)$, denoted by
$\|\cdot\|_{X'(0,1)}$, and defined, for $g \in  \M_+(0,1)$, by
\begin{equation}\label{E:assoc}
\|g\|_{X'(0,1)}=\sup_{\overset{f\in \M_+(0,1)}{\|f\|_{X(0,1)}\leq
1}}\int_0\sp1 f(s)g(s)\,ds.
\end{equation}
It turns out that  $\|\cdot\|_{X'(0,1)}$ is also an rearrangement
invariant~function norm, which is called the~\textit{associate
function norm} of $\|\cdot\|_{X(0,1)}$. The rearrangement
invariant~space $X'(\Omega)$ built upon the function norm $\|\cdot
\|_{X'(0,1)}$ is called the \textit{associate space} of $X(\Omega)$.
Given an rearrangement-invariant~function norm $\|\cdot\|_{X(0,1)}$,
the \textit{H\"older inequality}
\begin{equation}\label{E:holder}
\int_{\Omega}|u(x)v(x)|\,dx\leq\|u\|_{X(\Omega)}\|v\|_{X'(\Omega)}
\end{equation}
holds for every $u \in X(\Omega)$ and $v \in X'(\Omega)$. For every
rearrangement-invariant~space $X(\Omega)$ the identity
$X''(\Omega)=X(\Omega)$ holds and, moreover, for every
$f\in\M(\Omega)$, we have that
\begin{equation}\label{E:dual-norm}
\|u\|_{X(\Omega)}=\sup_{v\in\M(\Omega);\
\|v\|_{X'(\Omega)}\leq1}\int_0\sp1|u(x)v(x)|\,dx.
\end{equation}
The fundamental functions of a rearrangement-invariant~space
$X(\Omega)$ and its associate space $X'(\Omega)$ satisfy
\begin{equation}\label{E:product-of-fundamental-functions}
\varphi_X(t)\varphi_{X'}(t)=t\quad \textup{for every}\ t\in(0,1).
\end{equation}
Since we are assuming that $\Omega$ has finite measure,
\begin{equation}\label{l1linf}
L^\infty (\Omega) \to X(\Omega) \to L^1(\Omega)
\end{equation}
for every rearrangement-invariant~space $X(\Omega)$.
\\
A basic property of rearrangements is the \textit{Hardy-Littlewood
inequality} which tells us that, if $u, v \in\M(\Omega)$,
then
\begin{equation}\label{E:HL}
\int _\Omega |u(x) v(x)| dx \leq \int _0^1 u^*(t) v^*(t) dt.
\end{equation}
A~key fact concerning rearrangement-invariant function norms is the
\textit{Hardy--Littlewood--P\'olya principle} which states that if,
for some $u,v\in\M(\Omega)$,
\begin{equation}\label{E:HLP}
\int_0\sp tu\sp*(s)\,ds\leq\int_0\sp tv\sp*(s)\,ds\ \hbox{for
$t\in(0,1)$,}
\end{equation}
then
\[
\|u\|_{X(\Omega)}\leq \|v\|_{X(\Omega)}
\]
for every rearrangement-invariant~function norm
$\|\cdot\|_{X(0,1)}$. Moreover,
\begin{equation}\label{E:HL-product}
\|uv\|_{X(\Omega)}\leq\|u\sp*v\sp*\|_{X(0,1)}
\end{equation}
for every rearrangement-invariant~function norm
$\|\cdot\|_{X(0,1)}$, and all functions $u,v\in\M(\Omega)$.
Inequality \eqref{E:HL-product} follows from the inequality
$$
\int_0\sp t(uv)\sp*(s)\,ds\leq \int_0\sp tu\sp*(s)v\sp*(s)\,ds,
$$
and  the Hardy--Littlewood--P\'olya principle.

We refer the reader to~\cite{BS} for proofs of the results recalled
above, and for a comprehensive treatment of rearrangement-invariant
spaces.

Let $1\leq p,q\le\infty$ and $\beta\in\R$. We define the functional
$\|\cdot\|_{L\sp{p,q;\beta}(0,1)}$
as
$$
\|f\|_{L\sp{p,q;\beta}(0,1)}=
\left\|s\sp{\frac{1}{p}-\frac{1}{q}}\log \sp
\beta\left(\tfrac{2}{s}\right) f^*(s)\right\|_{L\sp q(0,1)}
$$
for $f \in \M_+(0,1)$. If one of the conditions in~\eqref{E:lz_bfs}
is satisfied, then $\|\cdot\|_{L\sp{p,q;\beta}(0,1)}$ is equivalent
to a~rearrangement-invariant~function norm,   called a
\textit{Lorentz--Zygmund norm} (for details see e.g~\cite{BR},~\cite{glz}
or~\cite{fs-1}). The corresponding rearrangement-invariant~space
$L\sp{p,q;\beta}(\Omega)$ is called the \textit{Lorentz--Zygmund
space}. When $\beta=0$, the space $L\sp{p,q;0}(\Omega)$ is denoted
by $L\sp{p,q}(\Omega)$ and called \textit{Lorentz space}. It is
known (e.g.~\cite[Theorem~6.11]{glz} or~\cite[Theorem~9.6.13]{fs-1}) that
\begin{equation}\label{E:lz_assoc}
(L\sp{p,q;\beta})'(\Omega)
=
\begin{cases}
L\sp{p',q';-\beta}(\Omega)&\textup{if}\ p<\infty;\\
L\sp{(1,q';-\beta-1)}(\Omega)&\textup{if}\ p=\infty,\ 1\leq q<\infty,\ \beta+\frac 1q<0;\\
L\sp{1}(\Omega)&\textup{if}\ p=\infty,\ q=\infty,\ \beta=0,
\end{cases}
\end{equation}
where $L\sp{(p,q;\beta)}(\Omega)$ denotes  the function space
defined analogously to $L\sp{p,q;\beta}(\Omega)$ but with the
functional $\|\cdot\|_{L\sp{(p,q;\beta)}(0,1)}$ given by
$$
\|f\|_{L\sp{(p,q;\beta)}(0,1)}=
\left\|s\sp{\frac{1}{p}-\frac{1}{q}}\log \sp
\beta\left(\tfrac{2}{s}\right) f^{**}(s)\right\|_{L\sp q(0,1)}
$$
for $f \in \M_+(0,1)$. Note that, if $\beta =0$, then
$$ L\sp{p,q;0}(\Omega)= L\sp{p,q}(\Omega),$$
a standard Lorentz space. In particular, if $p=q$, then
$$L\sp{p,p}(\Omega) = L\sp{p}(\Omega),$$
 a Lebesgue space.

A generalization of the Lebesgue spaces in a different direction is
provided by the Orlicz spaces. Let $A: [0, \infty ) \to [0, \infty
]$ be a Young function, namely a convex (non-trivial),
left-continuous function vanishing at $0$.
 The Orlicz space $L^A (\Omega)$ is the
rearrangement-invariant space associated with  the
 \emph{Luxemburg function norm} defined  as
\begin{equation*}
\|f\|_{L^A(0,1)}= \inf \left\{ \lambda >0 :  \int_0\sp 1A \left(
\frac{f(s)}{\lambda} \right) ds \leq 1 \right\}
\end{equation*}
for $f \in \M _+(0,1)$. In particular, $L^A (\Omega)= L^p (\Omega)$
if $A(t)= t^p$ for some $p \in [1, \infty )$, and $L^A (\Omega)=
L^\infty (\Omega)$ if $A(t)= \infty \chi_{(1, \infty)}(t)$.
\\ The associate function norm of $\|\cdot\|_{L^A(0,1)}$ is equivalent to
the function norm $\|\cdot\|_{L^{\widetilde A}(0,1)}$, where
$\widetilde A$ is the Young conjugate of $A$ defined as
$$\widetilde A(t) = \sup\{ts- A(s): s \geq 0\} \qquad \hbox{for
$t\geq 0$.}$$

\section{Key one-dimensional inequalities}\label{S:pointwise-inequalities}

In this section we shall state and prove two key assertions concerning
one-dimensional inequalities involving non-increasing functions and
rearrangement-invariant~spaces defined on an interval. Both these
results are of independent interest and they constitute a new
approach to inequalities involving products of functions.

Assume that $I$ is a~positive non-decreasing function on $(0,1)$. We
denote by $H_I$ the operator defined at every nonnegative measurable
function $g$ on $(0,1)$ by
\[
H_Ig(t)=\int_t\sp 1\frac{g(r)}{I(r)}\,dr \quad \hbox{for
$t\in(0,1)$.}
\]
Moreover, given $m\in\N$, we set
\[
H_I\sp m=\underbrace{H_I\circ H_I\circ \dots\circ H_I}_{
m-\textup{times}}.
\]
We also denote by $H_I\sp 0$ the identity operator.
It is easily verified that
\[
H\sp m_Ig(t)=\frac1{(m-1)!}\int_t\sp
1\frac{g(s)}{I(s)}\left(\int_t\sp s\frac{dr}{I(r)}\right)\sp
{m-1}\,ds \quad \hbox{for $m\in\N$ and $t\in(0,1)$,}
\]
 see~\cite[Remark~8.2~(iii)]{CPS}.

Let $X(0,1)$ be a rearrangement-invariant~space and let $m
\in \N$. If the function $I$ satisfies~\eqref{E:2I}, then the optimal (smallest) rearrangement-invariant~space $X_m(0,1)$ such that
\begin{equation}\label{optHm}
H\sp m_{I} : X(0,1) \to X_m(0,1)
\end{equation}
is endowed with the function norm $\|\cdot\|_{X_m(0,1)}$, whose
associate function norm is given by
\begin{equation}\label{E:optimal-hardy-embedding}
\|g\|_{X_m'(0,1)}=\left\|\frac{1}{I(t)}\int_0\sp
tg\sp*(s)\left(\int_s\sp
t\frac{dr}{I(r)}\right)\sp{m-1}\,ds\right\|_{X'(0,1)}
\end{equation}
for $g\in\M(0,1)$  \cite[Proposition~8.3]{CPS}.

\bigskip

The following lemma  provides us with a  pointwise inequality
involving the operator $H\sp k_I$ for $k=1,\dots,m-1$ for $I$
satisfying~\eqref{E:convergence-of-I-omegaI}. For every such $I$ we
denote by ${\psi _I}$ the function defined by
\begin{equation}\label{E:psi}
{\psi _I}(t)=\left(\int_0\sp t\frac{dr}{I(r)}\right)\sp{m} \quad
\hbox{for $t\in(0,1)$.}
\end{equation}

\begin{lemma}\label{T:thm1}
Assume that $m, k\in\N$, $m\geq 2$, $1 \leq k \leq m-1$. Let $I$ be
positive a~non-decreasing function on $(0,1)$
satisfying~\eqref{E:convergence-of-I-omegaI} and let ${\psi _I}$ be
the function defined by~\eqref{E:psi}. There exists a constant
$C=C(m)$ such that, if $g\in\mathcal M_+(0,1)$ and
\begin{equation}\label{E:pointwise-estimate-for-function}
g\sp*(t)\leq\frac1{{\psi _I}(t)}\quad\textit{for }\ t\in(0,1),
\end{equation}
then
\begin{equation}\label{E:pointwise-estimate-for-hardy-operator}
H\sp k_Ig\sp*(t)\leq C  g\sp*(t)\sp{1-\frac km}\quad \textit{for }\
t\in(0,1)\ \textit{and}\ k\in\{1,\dots,m\}.
\end{equation}
\end{lemma}

\begin{proof}
Fix $g\in\mathcal M_+(0,1)$ such
that~\eqref{E:pointwise-estimate-for-function} holds, and any
$t\in(0,1)$. Define $a\in(0,1]$ by the identity
\[
g\sp*(t)=\frac1{{\psi _I}(a)}\quad\textup{if}\ g\sp*(t)>\frac1{{\psi
_I}(1)}
\]
and $a=1$ otherwise. Note that the definition is correct since
${\psi _I}$ is continuous and strictly increasing on $(0,1)$, and
 $\lim_{t\to0_+}{\psi _I}(t)=0$.
\par\noindent
Assume first that $g\sp*(t)>\frac 1{{\psi _I}(1)}$.
Then~\eqref{E:pointwise-estimate-for-function} and the monotonicity
of ${\psi _I}$ imply that $t\leq a\leq 1$. We thus get
\[
g\sp*(s)\leq
\begin{cases}
g\sp*(t)&\textup{if}\ s\in[t,a),\\
\frac1{{\psi _I}(s)}&\textup{if}\ s\in[a,1).
\end{cases}
\]
Fix $k\in\{1,\dots, m-1\}$. Then, consequently,
\begin{align*}
(k-1)!H\sp k_Ig\sp*(t)
&=
\int_t\sp a\frac{g\sp*(s)}{I(s)}\left(\int_t\sp s\frac{dr}{I(r)}\right)\sp{k-1}\,ds+
\int_a\sp1\frac{g\sp*(s)}{I(s)}\left(\int_t\sp s\frac{dr}{I(r)}\right)\sp{k-1}\,ds\\
&\leq g\sp*(t)\int_t\sp a\frac{1}{I(s)}\left(\int_t\sp
s\frac{dr}{I(r)}\right)\sp{k-1}\,ds+ \int_a\sp1\frac{1}{{\psi
_I}(s)I(s)}\left(\int_t\sp s\frac{dr}{I(r)}\right)\sp{k-1}\,ds.
\end{align*}
By the definition of $a$,
\[
g\sp*(t)\int_t\sp a\frac{1}{I(s)}\left(\int_t\sp
s\frac{dr}{I(r)}\right)\sp{k-1}\,ds=\frac{g\sp*(t)}{k}\left(\int_t\sp
a\frac{dr}{I(r)}\right)\sp{k} \leq \frac{g\sp*(t)}{k}\left(\int_0\sp
a\frac{dr}{I(r)}\right)\sp{k} = \frac{g\sp*(t)}{k}{\psi
_I}(a)\sp{\frac km} = \frac1k g\sp*(t)\sp{1-\frac km}
\]
and
\begin{align*}
\int_a\sp1\frac{1}{{\psi _I}(s)I(s)}\left(\int_t\sp
s\frac{dr}{I(r)}\right)\sp{k-1}\,ds &=
\int_a\sp1\frac{1}{I(s)}\frac{\left(\int_t\sp
s\frac{dr}{I(r)}\right)\sp{k-1}}{\left(\int_0\sp
s\frac{dr}{I(r)}\right)\sp{m}}\,ds \leq
\int_a\sp1\frac{1}{I(s)}\left(\int_0\sp s\frac{dr}{I(r)}\right)\sp{k-m-1}\,ds\\
&=
\tfrac1{m-k}\left[\left(\int_0\sp a\frac{dr}{I(r)}\right)\sp{k-m}-\left(\int_0\sp 1\frac{dr}{I(r)}\right)\sp{k-m}\right]\\
&\leq \tfrac1{m-k}\left(\int_0\sp a\frac{dr}{I(r)}\right)\sp{k-m} =
\tfrac1{m-k}{\psi _I}(a)\sp{\frac km-1} =
\tfrac1{m-k}g\sp{*}(t)\sp{1-\frac k{m}}.
\end{align*}
Assume next that $g\sp*(t)\leq \frac 1{{\psi _I}(1)}$. Then $a=1$.
Similarly as above, we have that
\begin{align*}
(k-1)!H\sp k_Ig\sp*(t)
&=
\int_t\sp 1\frac{g\sp*(s)}{I(s)}\left(\int_t\sp s\frac{dr}{I(r)}\right)\sp{k-1}\,ds
\leq
g\sp*(t)\int_t\sp 1\frac{1}{I(s)}\left(\int_t\sp s\frac{dr}{I(r)}\right)\sp{k-1}\,ds\\
&= \frac{g\sp*(t)}{k}\left(\int_t\sp 1\frac{dr}{I(r)}\right)\sp{k}
\leq \frac{g\sp*(t)}{k}{\psi _I}(1)\sp{\frac km} \leq
\frac1{k}g\sp*(t)\sp{1-\frac km}.
\end{align*}
Altogether, inequality
\eqref{E:pointwise-estimate-for-hardy-operator} follows.
\end{proof}

Given   a rearrangement-invariant~function norm $\|\cdot\|_{X(0,1)}$ and
$p\in(1,\infty)$, we define the functional
$\|\cdot\|_{X\sp p(0,1)}$ by
\[
\|g\|_{X\sp p(0,1)}=\|g \sp{p}\|_{X(0,1)}\sp{\frac 1p}\quad
\hbox{for $g\in \M_+(0,1)$.}
\]
The functional $\|\cdot\|_{X\sp p(0,1)}$ is also an rearrangement
invariant~function norm. Moreover, the inequality
\begin{equation}\label{E:generalized-hoelder}
\|fg\|_{X(0,1)}\leq \|f\|_{X\sp{p}(0,1)}\|g\|_{X\sp{p'}(0,1)}
\end{equation}
holds for every $f,g\in\M(0,1)$ (see e.g.~\cite[Lemma~1]{MP}).

The following lemma, of possible independent interest, is a major
tool in the proofs of our main results.

\begin{lemma}\label{T:thm2}
Assume that $m\in\N$, $m\geq 2$. Let $I$ be a positive
non-decreasing function on $(0,1)$, and let $\psi _I$ be the
function defined by \eqref{E:psi}. Let $\|\cdot\|_{X(0,1)}$ be
a~rearrangement-invariant~function norm. Then the following
statements are equivalent:

\textup{(i)} Condition~\eqref{E:convergence-of-I-omegaI} holds, and
there exists a positive constant $C$ such that
\begin{equation}\label{E:embedding-to-marcinkiewicz}
\sup_{t\in(0,1)}g\sp{**}(t){\psi _I}(t)\leq C\|g\|_{X(0,1)}
\end{equation}
 for every
$g\in\M_+(0,1)$.

\textup{(ii)} Condition~\eqref{E:convergence-of-I-omegaI} holds, and
there exists a positive constant $C$ such that
\begin{equation}\label{E:embedding-to-weak-marcinkiewicz}
\sup_{t\in(0,1)}g\sp*(t){\psi _I}(t)\leq C\|g\|_{X(0,1)}
\end{equation}
 for every
$g\in\M_+(0,1)$.

\textup{(iii)} For every $k\in\N$, $1\leq k\leq m-1$, there is
a~positive constant $C$ such that
\begin{equation}\label{E:relation-between-x-and-y}
\|g\|_{X\sp {\tfrac m{m-k}}(0,1)}\leq C\|g\|_{X_k(0,1)}
\end{equation}
 for every
$g\in\M_+(0,1)$.

\textup{(iv)} For every $k\in\N$, $1\leq k\leq m-1$, there exists
a~positive constant $C$ such that
\begin{equation}\label{E:estimate-of-product}
\|fg\|_{X(0,1)}\leq C\|f\|_{X_k(0,1)}\|g\|_{X_{m-k}(0,1)}
\end{equation}
 for every
$f,g\in\M_+(0,1)$.

 \textup{(v)} There exists $k\in\N$, $1\leq k\leq
m-1$, and a~positive constant $C$ such that
\eqref{E:estimate-of-product} holds.

\textup{(vi)} There exists $k\in\N$, $1\leq k\leq m-1$, and
a~positive constant $C$ such that
\begin{equation}\label{E:two-hardys}
\|H\sp k_If \, H\sp{m-k}_Ig\|_{X(0,1)}\leq C \|f\|_{X(0,1)}
\|g\|_{X(0,1)}
\end{equation}
for every $f,g\in\M_+(0,1)$.

 \textup{(vii)} There exists a positive
constant $C$ such that
\begin{equation}\label{E:fundamental-estimate}
\left(\int_0\sp t\frac{ds}{I(s)}\right)\sp m\leq C\varphi_X(t) \quad
\hbox{for  $ t\in(0,1)$.}
\end{equation}
\end{lemma}

\begin{proof}
(i)$\Rightarrow$(ii)\quad This implication is trivial thanks to the
universal pointwise estimate $g\sp*(t)\leq g\sp{**}(t)$ which holds
for every $g\in\M(0,1)$ and every $t\in(0,1)$.

(ii)$\Rightarrow$(iii)\quad Fix $k\in \N$, $1\leq k\leq m-1$. In
view of the optimality of the space $X_k(0,1)$ in $H\sp
k_I:X(0,1)\to X_k(0,1)$, mentioned above Lemma~\ref{T:thm1}, the assertion will follow once
we show that
\begin{equation}\label{Hbound} H\sp{k}_I : X(0,1) \to
X\sp{\frac m{m-k}}(0,1).
\end{equation}
\par\noindent
 Let $g\in X(0,1)$, $g\not\equiv0$, and let $C$ be
the constant from~\eqref{E:embedding-to-weak-marcinkiewicz}. Define
$h=\frac g{C\|g\|_{X(0,1)}}$.
Inequality~\eqref{E:embedding-to-weak-marcinkiewicz} implies that
$h\sp*(t)\leq\frac 1{{\psi _I}(t)}$ for $t\in(0,1)$, whence, by
Lemma~\ref{T:thm1},
\[
H\sp{k}_Ih\sp*(t)\leq C'h\sp*(t)\sp{1-\frac km} \quad \hbox{for
$t\in(0,1)$,}
\]
for some constant $C'=C'(m,k)$. Thus,
\begin{align*}
\|H\sp{k}_Ig\sp*\|_{X\sp{\frac{m}{m-k}}(0,1)}
&=
C\|g\|_{X(0,1)}\|H\sp{k}_Ih\sp*\|_{X\sp{\frac{m}{m-k}}(0,1)}
\leq C'C\|g\|_{X(0,1)}\|{h\sp*}\sp{1-\frac km}\|_{X\sp{\frac{m}{m-k}}(0,1)}\\
&=C'C\|g\|_{X(0,1)}\|h\sp*\|_{X(0,1)}\sp{1-\frac km} =
 C'C\sp{\frac km}\|g\|_{X(0,1)}.
\end{align*}
By~\cite[Corollary~9.8]{CPS}, this is equivalent to the existence of
a~positive constant $C(m,k,X)$ such that
\[
\|H\sp{k}_Ig\|_{X\sp{\frac{m}{m-k}}(0,1)} \leq C(m,k,X)
\|g\|_{X(0,1)}
\]
for every $g\in \M_+(0,1)$. Hence, \eqref{Hbound} follows.

(iii)$\Rightarrow$(iv)\quad Fix $k\in\{1,\dots,m-1\}$ and let $f,g\in\M_+(0,1)$.
On applying first ~\eqref{E:relation-between-x-and-y} to $f$ in place of $g$,  and then~\eqref{E:relation-between-x-and-y} again, this time with $k$ replaced by $m-k$, we obtain
\[
\|f\|_{X\sp{\tfrac {m}{m-k}}(0,1)}\leq C\|f\|_{X_k(0,1)}
\quad \textup{and}\quad \|g\|_{X\sp{\tfrac mk}(0,1)}\leq C\|g\|_{X_{m-k}(0,1)}.
\]
Combining these estimates with~\eqref{E:generalized-hoelder}, with
$p=\frac{m}{m-k}$, yields
\[
\|fg\|_{X(0,1)}\leq C\sp2\|f\|_{X_k(0,1)}\|g\|_{X_{m-k}(0,1)},
\]
and (iv) follows.

(iv)$\Rightarrow$(v)\quad This implication is trivial.

(v)$\Rightarrow$(vi)\quad Let $k\in\N$, $1\leq k\leq m-1$, be such
that~\eqref{E:estimate-of-product} holds, and let $f,g\in\M_+(0,1)$.
On making use of~\eqref{E:estimate-of-product} with $H\sp k_If$ and
$H\sp{m-k}_Ig$ in the place of $f$ and $g$, respectively, we obtain
\[
\|H\sp k_If \, H\sp{m-k}_Ig\|_{X(0,1)}\leq C\|H\sp k_If\|_{X_k(0,1)}
\|H\sp{m-k}_Ig\|_{X_{m-k}(0,1)}.
\]
It follows from~\eqref{optHm} that
\[
H\sp k_I:X(0,1)\to X_k(0,1)\quad\textup{and}\quad H\sp{m-k}_I:X(0,1)\to X_{m-k}(0,1).
\]
Coupling these facts with the preceding inequality implies that
\[
\|H\sp k_If \, H\sp{m-k}_Ig\|_{X(0,1)}\leq C' \|f\|_{X(0,1)}
\|g\|_{X(0,1)}
\]
for some positive constant $C'=C'(m,k,I,X)$ but independent of
$f,g\in\M_+(0,1)$. Thus, the property (vi) follows.

(vi)$\Rightarrow$(vii)\quad  Let $k\in\N$, $1\leq k\leq m-1$, be
such that~\eqref{E:two-hardys} holds. Assume, for the time being,
that $m<2k$. On replacing, if necessary, $\|\cdot\|_{X(0,1)}$ with
the equivalent norm $C\|\cdot\|_{X(0,1)}$, we may suppose, without
loss of generality, that  $C=1$ in~\eqref{E:two-hardys}. Thus,
\begin{equation}\label{E:phu1}
\|(H\sp k_If)(H\sp{m-k}_Ig)\|_{X(0,1)}\leq \|f\|_{X(0,1)}
\|g\|_{X(0,1)}
\end{equation}
for every $f,g\in\M_+(0,1)$.
\par\noindent
Let $b>-1$. Fix $a\in(0,1]$ and $\ve\in(0,a)$. Set
\begin{equation}\label{E:def-f-g}
f(t)=\chi_{(\ve,a)}(t)\left(\int_t\sp a\frac{dr}{I(r)}\right)\sp b,
\quad g(t)=\chi_{(\ve,a)}(t)\left(\int_t\sp
a\frac{dr}{I(r)}\right)\sp{(b+k)\frac{k}{m-k}+k-m}
\end{equation}
for $t\in(0,1)$. One can verify that
\begin{equation}\label{E:fx}
H\sp k_{I}f(t)=\frac{1}{(b+1)\dots(b+k)}\left(\int_t\sp
a\frac{dr}{I(r)}\right)\sp{(b+k)} \quad \hbox{for $t\in(\ve,a)$.}
\end{equation}
Since we are assuming that  $m<2k$ and $b>-1$,
\begin{equation}\label{E:dis}
(b+k)\tfrac{k}{m-k}+k-m>(-1+k)\tfrac{k}{2k-k}+k-2k=-1.
\end{equation}
Hence, analogously to \eqref{E:fx},
\begin{equation}\label{E:gx}
H\sp{m-k}_{I}g(t)=\frac{1}{[(b+k)\frac{k}{m-k}+k-m+1]\dots[(b+k)\frac{k}{m-k}]}\left(\int_t\sp
a\frac{dr}{I(r)}\right)\sp{(b+k)\frac{k}{m-k}} \quad \hbox{for
$t\in(\ve,a)$.}
\end{equation}
Note that
\begin{equation}\label{E:tagg}
(b+k)\tfrac{k}{m-k}+k-m=b\tfrac{m-k}{k}+(b+k)\tfrac{m}{m-k}\tfrac{2k-m}{k}.
\end{equation}
Set $p=\frac{k}{m-k}$. The assumption $m<2k$ guarantees that $p>1$.
Moreover, $p'=\frac{k}{2k-m}$. Thus, inequality
~\eqref{E:generalized-hoelder}, applied to this choice of $p$, and
equation \eqref{E:tagg} tell us that
\begin{align}
\|g\|_{X(0,1)}
&=
\left\|\chi_{(\ve,a)}(t)\left(\int_t\sp a\frac{dr}{I(r)}\right)\sp{(b+k)\frac{k}{m-k}+k-m}\right\|_{X(0,1)}\label{E:hx}\\
&\leq
\left\|\chi_{(\ve,a)}(t)\left(\int_t\sp a\frac{dr}{I(r)}\right)\sp b\right\|_{X(0,1)}\sp{\frac{m-k}k}
 \left\|\chi_{(\ve,a)}(t)\left(\int_t\sp a\frac{dr}{I(r)}\right)\sp {(b+k)\frac{m}{m-k}}\right\|_{X(0,1)}\sp{\frac{2k-m}k}\nonumber\\
&= \|f\|_{X(0,1)}\sp{\frac mk-1}
\left\|\chi_{(\ve,a)}(t)\left(\int_t\sp a\frac{dr}{I(r)}\right)\sp
{(b+k)\frac{m}{m-k}}\right\|_{X(0,1)}\sp{\frac{2k-m}k}\nonumber.
\end{align}
Coupling ~\eqref{E:phu1} with~\eqref{E:hx} yields
\begin{equation}\label{E:pja}
\|H\sp k_If \, H\sp{m-k}_Ig\|_{X(0,1)}\leq
\left\|f\right\|_{X(0,1)}\sp{\frac mk}
\left\|\chi_{(\ve,a)}(t)\left(\int_t\sp a\frac{dr}{I(r)}\right)\sp
{(b+k)\frac{m}{m-k}}\right\|_{X(0,1)}\sp{\frac{2k-m}k}.
\end{equation}
By~\eqref{E:fx} and~\eqref{E:gx},
\[
H\sp k_If(t) \,H\sp{m-k}_Ig(t)=
\frac{1}{(b+1)\dots(b+k)[(b+k)\frac{k}{m-k}+k-m+1]\dots[(b+k)\frac{k}{m-k}]}\left(\int_t\sp
a\frac{dr}{I(r)}\right)\sp{(b+k)\frac{m}{m-k}}
\]
for  $t\in(\ve,a)$. On setting
\begin{equation}\label{E:B}
B(b)=(b+1)\dots(b+k)\big[(b+k)\tfrac{k}{m-k}+k-m+1]\dots[(b+k)\tfrac{k}{m-k}\big]
\end{equation}
and making use of~\eqref{E:pja}, we obtain
\begin{align}
\frac1{B(b)}\left\|\chi_{(\ve,a)}(t)\left(\int_t\sp a\frac{dr}{I(r)}\right)\sp{(b+k)\frac{m}{m-k}}\right\|_{X(0,1)}
&=
\left\|\chi_{(\ve,a)} \,H\sp k_If \, H\sp{m-k}_Ig \right\|_{X(0,1)}\label{E:wu}\\
&\leq
\left\|H\sp k_If \, H\sp{m-k}_Ig\right\|_{X(0,1)}\nonumber\\
&\leq
\left\|f\right\|_{X(0,1)}\sp{\frac mk}
\left\|\chi_{(\ve,a)}(t)\left(\int_t\sp a\frac{dr}{I(r)}\right)\sp {(b+k)\frac{m}{m-k}}\right\|_{X(0,1)}\sp{\frac{2k-m}k}\nonumber.
\end{align}
The function
\[
(0,1) \ni t\mapsto \chi_{(\ve,a)}(t)\left(\int_t\sp
a\frac{dr}{I(r)}\right)\sp{(b+k)\frac{m}{m-k}}
\]
is bounded and hence, by~\eqref{l1linf}, belongs to $X(0,1)$. Thus,
\[
\left\|\chi_{(\ve,a)}(t)\left(\int_t\sp a\frac{dr}{I(r)}\right)\sp {(b+k)\frac{m}{m-k}}\right\|_{X(0,1)}<\infty.
\]
Moreover, $1-\frac{2k-m}{k}=\frac mk-1$. Therefore,~\eqref{E:wu} yields
\[
\left\|\chi_{(\ve,a)}(t)\left(\int_t\sp a\frac{dr}{I(r)}\right)\sp{(b+k)\frac{m}{m-k}}\right\|_{X(0,1)}\sp{\frac mk-1}
\leq B(b)
\left\|f\right\|_{X(0,1)}\sp{\frac mk}.
\]
On raising this inequality to the power $\frac km$, and recalling
the definition of $f$, we get
\begin{equation}\label{E:star3}
\left\|\chi_{(\ve,a)}(t)\left(\int_t\sp a\frac{dr}{I(r)}\right)\sp{(b+k)\frac{m}{m-k}}\right\|_{X(0,1)}\sp{1-\frac km}
\leq B(b)\sp{\frac km}
\left\|\chi_{(\ve,a)}(t)\left(\int_t\sp a\frac{dr}{I(r)}\right)\sp b\right\|_{X(0,1)}.
\end{equation}
\par\noindent
Next, assume that $m=2k$. Note that \eqref{E:dis}  holds also in
this case. Let  $f$ and $g$ be defined by~\eqref{E:def-f-g} again.
Then
\[
(b+k)\tfrac{k}{m-k}+k-m=b,
\]
whence $f=g$. Moreover, since $k=m-k$, we also have that $H\sp
k_If=H\sp{m-k}_Ig$. Therefore,~\eqref{E:phu1},~\eqref{E:fx}
and~\eqref{E:gx} imply that
\[
\left\|\chi_{(\ve,a)}(t)\left(\int_t\sp a\frac{dr}{I(r)}\right)\sp{2(b+k)}\right\|_{X(0,1)}\sp{\frac12}
\leq
(b+1)\dots(b+k)
\left\|\chi_{(\ve,a)}(t)\left(\int_t\sp a\frac{dr}{I(r)}\right)\sp{b}\right\|_{X(0,1)}.
\]
The assumption $m=2k$ entails that $B(b)=[(b+1)\dots(b+k)]\sp2$, and
$\frac km=\frac 12$. Hence, \eqref{E:star3} follows. We have thus
established~\eqref{E:star3} whenever $m\leq 2k$. From now on, we
keep this assumption in force.
\par\noindent Define $b_0=0$ and, for
$j\in\N$,
\[
b_j=(b_{j-1}+k)\tfrac{m}{m-k},
\]
namely
\begin{equation}\label{E:bj}
b_j=m\big[(\tfrac{m}{m-k})\sp j-1\big].
\end{equation}
We next set, for $j\in\N$,
\[
K_j=\prod_{i=0}\sp{j-1}B(b_i)\sp{\frac km(\frac{m-k}{m})\sp i}.
\]
Let us note that the assumption $2k\geq m$ implies that
$B(b_j)\geq 1$ for $j\in\N\cup\{0\}$, and hence $K_j\geq 1$ as well.
\par\noindent
We claim that, for every $j\in\N$,
\begin{equation}\label{E:star-general-ve}
\left\|\chi_{(\ve,a)}(t)\left(\int_t\sp a\frac{dr}{I(r)}\right)\sp{b_j}\right\|_{X(0,1)}\sp{(\frac{m-k}{m})\sp j}
\leq K_j
\|\chi_{(\ve,a)}\|_{X(0,1)}.
\end{equation}
Indeed, choosing $b=0$ in~\eqref{E:star3}, yields
\eqref{E:star-general-ve} for $j=1$. Assume now
that~\eqref{E:star-general-ve} holds for some fixed $j\in\N$. Then,
by~\eqref{E:star3} with $b=b_j$,
\[
\left\|\chi_{(\ve,a)}(t)\left(\int_t\sp a\frac{dr}{I(r)}\right)\sp{b_{j+1}}\right\|_{X(0,1)}\sp{(\frac{m-k}{m})\sp{j+1}}
\leq B(b_j)\sp{\frac km(\frac{m-k}{m})\sp j}
\left\|\chi_{(\ve,a)}(t)\left(\int_t\sp a\frac{dr}{I(r)}\right)\sp{b_j}\right\|_{X(0,1)}\sp{(\frac{m-k}{m})\sp j}.
\]
Thus, by the induction assumption,
\[
\left\|\chi_{(\ve,a)}(t)\left(\int_t\sp a\frac{dr}{I(r)}\right)\sp{b_{j+1}}\right\|_{X(0,1)}\sp{(\frac{m-k}{m})\sp{j+1}}
\leq B(b_j)\sp{\frac km(\frac{m-k}{m})\sp j}K_j
\left\|\chi_{(\ve,a)}\|_{X(0,1)}=K_{j+1}\|\chi_{(\ve,a)}\right\|_{X(0,1)},
\]
and~\eqref{E:star-general-ve} follows.
\par\noindent
Letting $\ve\to0^+$ in \eqref{E:star-general-ve} and making use of
property (P3) of rearrangement-invariant function norms yields
\begin{equation}\label{E:star-general}
\left\|\chi_{(0,a)}(t)\left(\int_t\sp a\frac{dr}{I(r)}\right)\sp{b_j}\right\|_{X(0,1)}\sp{(\frac{m-k}{m})\sp j}
\leq K_j
\|\chi_{(0,a)}\|_{X(0,1)}.
\end{equation}
\par\noindent
Fix $j\in\N$, and set
\[
g_j(t)=\chi_{(0,a)}(t)\left(\int_t\sp
a\frac{dr}{I(r)}\right)\sp{b_{j}} \quad \hbox{for $t \in (0,1)$.}
\]
Then, owing to~\eqref{E:star-general}
\begin{equation}\label{E:pess}
\|g_j\|_{X(0,1)}\sp{(\frac{m-k}{m})\sp{j}}\leq K_j\|\chi_{(0,a)}\|_{X(0,1)}.
\end{equation}
By the H\"older inequality
and~\eqref{E:product-of-fundamental-functions},
\[
\int_0\sp a g_j(t)\,dt
\leq
\|g_j\|_{X(0,1)}\|\chi_{(0,a)}\|_{X'(0,1)}
=
\frac{a\|g_j\|_{X(0,1)}}{\|\chi_{(0,a)}\|_{X(0,1)}}.
\]
By~\eqref{E:pess},
\begin{align*}
\frac 1a\int_0\sp a g_j(t)\,dt
\leq K_j\sp{\frac1{(\frac{m-k}{m})\sp
j}}\|\chi_{(0,a)}\|_{X(0,1)}\sp{\frac{1-(\frac{m-k}{m})\sp
j}{(\frac{m-k}{m})\sp j}}.
\end{align*}
We have thus shown that
\begin{equation}\label{E:psi1}
\left(\frac 1a\int_0\sp a
g_j(t)\,dt\right)\sp{\frac{(\frac{m-k}{m})\sp
j}{1-(\frac{m-k}{m})\sp j}} \leq
K_j\sp{\frac{1}{1-(\frac{m-k}{m})\sp j}}\|\chi_{(0,a)}\|_{X(0,1)}
\end{equation}
for  $j\in\N$.
\par\noindent
Observe that
\begin{align}
\lim_{j\to\infty}\left(\frac 1a\int_0\sp a g_j(t)\,dt\right)\sp{\frac{(\frac{m-k}{m})\sp j}{1-(\frac{m-k}{m})\sp j}}
&=
\lim_{j\to\infty}\left(\frac 1a\int_0\sp a \left(\int_t\sp a\frac{dr}{I(r)}\right)\sp{b_j}\,dt\right)
\sp{\frac{1}{(\frac{m}{m-k})\sp j-1}}\label{E:psi2}\\
&=
\lim_{j\to\infty}\left(\frac 1a\int_0\sp a \left(\int_t\sp a\frac{dr}{I(r)}\right)\sp{m((\frac{m}{m-k})\sp j-1)}\,dt\right)\sp{\frac{1}{(\frac{m}{m-k})\sp j-1}}\nonumber\\
&=
\left\|\chi_{(0,a)}(t)\left(\int_t\sp a\frac{dr}{I(r)}\right)\sp{m}\right\|_{L\sp{\infty}(0,1)}
=\left(\int_0\sp a\frac{dr}{I(r)}\right)\sp{m}.\nonumber
\end{align}
On the other hand, since $\frac{m-k}m<1$,
\[
1-\left(\tfrac{m-k}m\right)\sp j\geq 1-\tfrac{m-k}m=\tfrac km
\]
 for  $j\in\N$. As observed above, $K_j\geq 1$ for every $j\in\N$.
By~\eqref{E:B}, if $b>-1$, then
\[
B(b)\leq(b+k)\sp{k}((b+k)\tfrac
k{m-k})\sp{m-k}\leq(b+k)\sp{m}(\tfrac m{m-k})\sp{m}.
\]
With the  choice $b=b_j$,  the last chain and~\eqref{E:bj} yield
\[
B(b_j)\leq (m((\tfrac{m}{m-k})\sp j-1)+k)\sp{m}(\tfrac
m{m-k})\sp{m}\leq (m(\tfrac{m}{m-k})\sp{j+1})\sp{m}
\]
for  $j\in\N\cup\{0\}$. Altogether, we deduce that
\[
K_j\sp{\frac{1}{1-(\frac{m-k}{m})\sp j}}\leq K_j\sp{\frac mk}
=\prod_{i=0}\sp{j-1}B(b_i)\sp{(\frac{m-k}{m})\sp i} \leq
\prod_{i=0}\sp{\infty}\left(\left(m\left(\tfrac{m}{m-k}\right)\sp
{i+1}\right)\sp m\right)\sp{(\frac{m-k}{m})\sp i}<\infty.
\]
On setting
\[
K=\prod_{i=0}\sp{\infty}\left(\left(m\left(\tfrac{m}{m-k}\right)\sp
{i+1}\right)\sp m\right)\sp{(\frac{m-k}{m})\sp i},
\]
we have that
\begin{equation}\label{E:psi3}
\limsup_{j\to\infty}K_j\sp{\frac{1}{1-(\frac{m-k}{m})\sp j}}\leq K.
\end{equation}
Therefore, combining ~\eqref{E:psi1}, \eqref{E:psi2}
and~\eqref{E:psi3} tells us that
\[
\left(\int_0\sp a\frac{dr}{I(r)}\right)\sp{m}\leq K \|\chi_{(0,a)}\|_{X(0,1)}.
\]
Hence, inequality~\eqref{E:fundamental-estimate} follows. Thus,
property  (vii) is proved when $m\leq 2k$. However, if this is not
the case, then  $m\leq2(m-k)$. The same argument as above, applied
with  $m-k$ in the place of $k$, leads to the conclusion.

(vii)$\Rightarrow$(i)\quad Let $g\in\M(0,1)$. Then,
by~\eqref{E:fundamental-estimate},
\[
\sup_{t\in(0,1)}g\sp{**}(t){\psi _I}(t)\leq
C\sup_{t\in(0,1)}g\sp{**}(t)\varphi_X(t).
\]
One has that
\[
\sup_{t\in(0,1)}g\sp{**}(t)\varphi_X(t)
\leq \|g\|_{X(0,1)}
\]
for every $g\in\M(0,1)$ (see e.g. It is a~classical fact (see
e.g.~\cite[Chapter~2, Proposition~5.9]{BS}). Combining the last two
estimates yields
\[
\sup_{t\in(0,1)}g\sp{**}(t){\psi _I}(t)\leq C\|g\|_{X(0,1)}
\]
for $ g\in\M(0,1)$, namely \eqref{E:embedding-to-marcinkiewicz}.
\end{proof}

{We conclude this section by showing
that assumption \eqref{E:psi-in-associate} is actually essentially
stronger than \eqref{E:fundamental-estimate-first}. }

\begin{prop}\label{R:psi-implies-fundamental}
{Let $m,n\in\N$, $n\geq 2$, $m\geq2$,
and let $\|\cdot\|_{X(0,1)}$ be a~rearrangement-invariant function
norm. Assume that $I$ is a positive non-decreasing function on
$(0,1)$. If ~\eqref{E:psi-in-associate} holds, then
~\eqref{E:fundamental-estimate-first} holds as well.}
\end{prop}
\begin{proof}
{By the H\"older inequality in
rearrangement-invariant spaces,
\begin{align*}
\bigg(\int_0\sp t \frac{ds}{I(s)}\bigg)\sp m & = m\int_0\sp
t\frac{1}{I(s)}\left(\int_0\sp
s\frac{dr}{I(r)}\right)\sp{m-1}\,ds\leq
m\left\|\frac1{I(s)}\left(\int_0\sp
s\frac{dr}{I(r)}\right)\sp{m-1}\right\|_{X'(0,1)}\|\chi_{(0,t)}\|_{X(0,1)}
\\ & = m \left\|\frac1{I(s)}\left(\int_0\sp
s\frac{dr}{I(r)}\right)\sp{m-1}\right\|_{X'(0,1)}\varphi_X(t) \qquad
\hbox{for $t \in (0, 1)$.}
\end{align*}
Hence, the assertion follows.}
\end{proof}

\begin{remark}\label{stronger}
It is easily  seen
that~\eqref{E:psi-in-associate} is in fact essentially stronger
than~\eqref{E:fundamental-estimate-first}, in general. Indeed, let
$I(t)=t\sp{\alpha}$ for some $\alpha\in\R$ such that
$\alpha\geq\frac1{n'}$ and $\alpha>\frac1{m'}$, and let
$\|\cdot\|_{X(0,1)}=\|\cdot\|_{L\sp q(0,1)}$ with
$q=\frac1{m(1-\alpha)}$. Then~\eqref{E:fundamental-estimate-first}
holds but~\eqref{E:psi-in-associate} does not. In other words, by
Theorems~\ref{T:algebra-for-W} and~\ref{T:algebra-restricted},
$\mathcal V\sp m L\sp{\frac1{m(1-\alpha)}}(\Omega)$ is not a~Banach
algebra for every $\Omega\in \mathcal J_{\alpha}$, yet it
satisfies~\eqref{E:corollary-lower} for every $\Omega\in \mathcal
J_{\alpha}$.
\end{remark}

\section{Proofs of the main results}\label{S:proof-of-main-theorem}

Here, we accomplish the proofs of the   results stated  in Section
\ref{S:intro}. \\ A result to be exploited in our proofs is an
embedding theorem for the space $\VV\sp mX(\Omega)$, which tells us
that, under  assumption~\eqref{E:2I},
\begin{equation}\label{E:optimal-Sobolev-embedding}
V\sp mX(\Omega)\to X_m(\Omega)
\end{equation}
where $ X_m(\Omega)$ is the rearrangement-invariant space built upon
the function norm $\|\cdot\|_{X_m(0,1)}$ given by \eqref{optHm}, and
that $X_m(\Omega)$ is the optimal (smallest) such
rearrangement-invariant space  \cite[Theorem~5.4]{CPS} (see also
~\cite{EKP}, \cite{Ci1} and \cite{T2} for earlier proofs in special
cases).

The next three lemmas are devoted to  certain ``worst possible''
domains whose isoperimetric function has a prescribed decay. Such
domains will be of use in the proof of the necessity of our
conditions in the main results.

\begin{lemma}\label{L:lemma}
Let $n\in\N$, $n\geq 2$, and let $I$ be a positive , non-decreasing
function satisfying~\eqref{E:I-divided-by-power-increasing}. Then
there exists a positive non-decreasing function $\widehat I$ in
$(0,1)$  such that $\widehat I\in C\sp 1(0,1)$, ${\widehat
I}\sp{n'}$ is convex on $(0,1)$, and
\begin{equation}\label{equival}
\widehat I (s)\approx I(s) \quad \hbox{for $s \in (0,1)$.}
\end{equation}
\end{lemma}

\begin{proof}
By~\eqref{E:I-divided-by-power-increasing},  there exists
a~non-decreasing function $\varsigma : (0,1) \to (0,\infty)$ such
that
\begin{equation}\label{equival0}
\frac{I(s)\sp{n'}}{s} \approx \varsigma (s) \qquad \hbox{for $s \in
(0, 1)$}.
\end{equation}
Set
\[
I_1(s)=\bigg(\int_0^s\varsigma (r)\,dr\bigg)^{1/n'} \qquad \hbox{for
$s \in (0, 1)$}.
\]
Then $I_1 \in C^0(0,1)$, and $I_1^{n'}$ is convex in $(0, 1)$.
Moreover, we claim that
\begin{equation}\label{equival2}
 I_1(s) \approx I(s) \qquad \hbox{for $s \in
(0, 1)$}.
\end{equation}
Indeed, by the monotonicity of $\varsigma$,
\begin{align}\label{equival3}
\tfrac 12 \varsigma (s/2) \leq \frac 1s \int _{s/2}^s\varsigma(r)dr
\leq \frac{I_1(s)^{n'}}{s} = \frac 1s \int _0^s \varsigma(r)dr  \leq
\varsigma (s) \qquad \hbox{for $s \in (0, 1)$}.
\end{align}
Equation \eqref{equival2} follows from \eqref{equival0} and
\eqref{equival3}. Similarly, on setting
$$
\widehat I (s) = \bigg(\int _0^s \frac{I _1(r)^{n'}}{r }
dr\bigg)^{1/n'} \qquad \hbox{for $s \in (0, 1)$},
$$
one has that $\widehat I \in C^1(0, 1)$, $\widehat I ^{n'}$ is
convex in $(0, 1)$, and
\begin{equation}\label{equival4}
\widehat I (s)\approx I_1 (s)  \qquad \hbox{for $s \in (0, 1)$.}
\end{equation}
Coupling \eqref{equival2} with \eqref{equival4} yields
\eqref{equival}. Thus, the function $\widehat I$ has the required
properties.
\end{proof}

\begin{lemma}\label{auxdiv}
Let $n\in\N$, $n\geq2$, and let  $I$ be a positive, non-decreasing
function on $(0,1)$
satisfying~\eqref{E:I-divided-by-power-increasing}.  Then there
exist $L \in (0, \infty]$ and a convex function $\eta : (0, L) \to
(0, \infty )$ such that the domain $\Omega_I \subset \R\sp n$,
defined by
\begin{equation}\label{omegadiv}
\Omega_I = \{ x\in\R\sp n\colon x=(x',x_n),\ x_n\in(0,L),\
x'\in\mathbb R^{n-1},\ |x'| < \eta (x_n)\},
\end{equation}
satisfies $|\Omega _I|=1$ and
\begin{equation}\label{auxdiv1}
I_{\Omega _I} (s) \approx I(s) \quad \hbox{for $s \in (0,\tfrac
12]$.}
\end{equation}
\end{lemma}

\begin{proof}
By Lemma~\ref{L:lemma}, we can assume with no loss of generality
that $I \in C^1(0, 1)$ and $I ^{n'}$ is convex in $(0, 1)$. Let $L
\in (0, \infty]$ be defined by
\begin{equation}\label{E:L}
L = \int _0^{1} \frac {dr}{I(r)},
\end{equation}
 and let  $M : [0,L) \to (0, 1]$ be the function implicitly
defined as
\begin{equation}\label{auxdiv4}
\int _{M(t)}^1 \frac {dr}{I(r)} = t \qquad \hbox{for $t \in [0, L
)$.}
\end{equation}
The function $M$ strictly decreases from $1$ to $0$. In particular,
$M$ is continuously differentiable in $(0, L)$, and
\begin{equation}\label{auxdiv5}
I (M(r)) = - M'(r) \qquad \hbox{for $r\in (0, L)$}\,.
\end{equation}
 On defining
 $\eta : (0, L) \to (0, \infty )$ as
\begin{equation}\label{eta}
\eta (r) = \bigg(\frac {I(M(r))}{\omega _{n-1}}\bigg)^{\frac 1{n-1}}
\qquad \hbox{for $r \in (0, L)$,}
\end{equation}
where $\omega_{n-1}$ is the volume of the $(n-1)$-dimensional unit
ball, we have that $\eta (r ) >0 $ for $r \in (0, L )$,  and, by
\eqref{auxdiv5},
\begin{equation}\label{auxdiv7}
|\{x \in \Omega _I: t < x_n < L\}|= \omega _{n-1} \int _t^L \eta
(r)^{n-1}\, dr = \int _t^L I(M(r))\, dr = \int _t^L - M'(r)\, dr =
M(t) \quad \hbox{for $t \in (0, L)$.}
\end{equation}
Moreover, $\eta$ is convex. To see this, notice that,
by~\eqref{auxdiv5},
\begin{align}\label{auxdiv6}
\omega _{n-1}^{\frac 1{n-1}} \eta '(r) = \big(I (M(r))^{\frac
1{n-1}}\big)'
& = \frac {1}{n-1}I '(M(r))I (M(r))^{\frac1{n-1}-1}M'(r)\\
\nonumber & =  - \frac {1}{n-1}I '(M(r))I (M(r))^{\frac 1{n-1}}
\quad \hbox{for $r \in (0, L)$.}
\end{align}
Thus, since $M(r)$ is decreasing, $\eta '(r)$ is increasing if and
only if $I '(s)I (s)^{\frac 1{n-1}}$ is increasing, and this is in
turn equivalent to the convexity of $I (s)^{n'}$. Equation
\eqref{auxdiv6} also tells us that
\[
\eta ' (0^+)= -\frac1{n-1}\omega_{n-1}^{-\frac
1{n-1}}I'(1^-)I(1^-)^{\frac 1{n-1}}
>- \infty.
\]
\\ By \eqref{auxdiv7}, with $t=0$, the set $\Omega _I$ as in \eqref{omegadiv},
with $\eta$ given by \eqref{eta}, satisfies $|\Omega _I|=1$.
Furthermore, owing to the properties of $\eta$, from either
\cite[Example~5.3.3.1]{Mabook} or \cite[Example~5.3.3.2]{Mabook},
according to whether $L<\infty$ or $L=\infty$, we  infer that there
exist positive constants $C_1$ and $C_2$ such that
\begin{equation}\label{auxdiv9}
C_1\, \eta (M^{-1}(s))^{n-1} \leq I_{\Omega _I} (s) \leq C_2\, \eta
(M^{-1}(s))^{n-1} \qquad \hbox{for $s \in (0, \frac 12].$}
\end{equation}
Hence, by \eqref{eta},
\begin{equation}\label{auxconv8}
\frac{C_1}{\omega _{n-1}}\, I(s) \leq I_{\Omega_I} (s) \leq
\frac{C_2}{\omega _{n-1}} \, I(s) \qquad \hbox{for $s \in (0, \frac
12],$}
\end{equation}
and \eqref{auxdiv1} follows.
\end{proof}

\begin{lemma}\label{R:remark}
Let $n\in\N$, $n\geq 2$. Let $I$ be a~positive non-decreasing
function on $(0,1)$ such that  $I \in C^1(0, 1)$ and $I ^{n'}$ is
convex in $(0, 1)$. Let $L$, $M$, $\eta$ and $\Omega_I$ be as in
Lemma \ref{auxdiv}.
Given $h\in\M_+(0,1)$, let $F:\Omega_I\to[0,\infty)$ be the function
defined by
\[
F(x)=h(M(x_n))\quad \textup{for}\ x\in\Omega_I.
\]
Then
\begin{equation}\label{E:starr}
F\sp*(t)=h\sp*(t)\quad \textup{for}\ t\in(0,1).
\end{equation}
\end{lemma}

\begin{proof}
On making use of~\eqref{eta}, of a change of variables and
of~\eqref{auxdiv5}, we obtain that
\begin{align*}
|\{x\in\Omega_I\colon F(x)>\lambda\}| &= |\{x\in\Omega_I\colon
h(M(x_n))>\lambda\}| =
\omega_{n-1}\int_{\{r\in(0,L)\colon h(M(r))>\lambda\}}\eta(r)\sp{n-1}\,dr\\
&=
\int_{\{r\in(0,L)\colon h(M(r))>\lambda\}}I(M(r))\,dr
=
\int_{\{t\in(0,1)\colon h(t)>\lambda\}}-\frac{I(t)}{M'(M\sp{-1}(t))}\,dt\\
&= |\{t\in(0,1)\colon h(t)>\lambda\}| \qquad \hbox{for $\lambda>0$.}
\end{align*}
Equation \eqref{E:starr} hence  follows, via the definition of the
decreasing rearrangement.
\end{proof}

\begin{prop}\label{T:embedding-to-linfty}
Assume that $m,n\in\N$, $n\geq 2$, and let $I$ be a positive
 non-decreasing function
satisfying~\eqref{E:I-divided-by-power-increasing}. Let
$\|\cdot\|_{X(0,1)}$ be a~rearrangement-invariant~function norm.
Assume that for every domain $\Omega\in\mathcal J_I$, the Sobolev
space $\VV\sp{m}X(\Omega)$ is a~Banach algebra. Then
condition~\eqref{E:psi-in-associate} holds. Moreover,
\begin{equation}\label{E:embedding-to-linfty}
\VV\sp{m}X(\Omega)\hra L\sp{\infty}(\Omega)
\end{equation}
for every such domain $\Omega$.
\end{prop}

\begin{proof}
By Lemma~\ref{L:lemma}, we can assume, without loss of generality,
that $I \in C^1(0, 1)$ and $I ^{n'}$ is convex in $(0, 1)$. Let $L$,
$M$, $\eta$ and $\Omega_I$ be as in Lemma~\ref{auxdiv}. Let
$f,g\in\M_+(0,1)$. We define the functions $u,v : \Omega_I \to
[0,\infty]$ by
\begin{equation}\label{E:def-u}
u (x) = \int _{M(x_n)} ^1 \frac{1}{I(r_1)} \int _{r_1} ^1
\frac{1}{I(r_2)} \dots \int _{r_{m-1}}^1  \frac{f(r_m)}{I(r_m)}
dr_m\, dr_{m-1}\, \dots dr_1 \quad \hbox{for $ x \in \Omega_I ,$}
\end{equation}
and
\begin{equation}\label{E:def-v}
v (x) = \int _{M(x_n)} ^1 \frac{1}{I(r_1)} \int _{r_1} ^1
\frac{1}{I(r_2)} \dots \int _{r_{m-1}}^1  \frac{g(r_m)}{I(r_m)}
dr_m\, dr_{m-1}\, \dots dr_1 \quad \hbox{for $ x \in \Omega_I.$}
\end{equation}
Then the functions $u$ and $v$ are $m$ times weakly differentiable
in $\Omega_I$. Since  $u$ is a non-decreasing function of the
variable $x_n$,
\begin{align*}
|\nabla u(x)|
&= \frac {\partial u}{\partial x_n}(x)=
-\frac{M'(x_n)}{I(M(x_n))}
\int _{M(x_n)}\sp 1
\frac{1}{I(r_{2})}\int_{r_2}\sp1  \dots \int _{r_{m-1}}^1 \frac{f(r_m)}{I(r_m)}
dr_m\, dr_{m-1}\, \dots dr_{2}\\
&= \int _{M(x_n)}\sp 1 \frac{1}{I(r_{2})}\int_{r_2}\sp1  \dots \int
_{r_{m-1}}^1 \frac{f(r_m)}{I(r_m)} dr_m\, dr_{m-1}\, \dots dr_{2}
\quad \hbox{for a.e. $ x \in \Omega_I$,}
\end{align*}
where the last equality holds by \eqref{auxdiv5}. Similarly, if $1
\leq k \leq m-1$,
\begin{equation}\label{Enec2}
|\nabla ^k u(x)| =\frac {\partial\sp k u}{\partial x_n\sp k}(x)=
\int _{M(x_n)}\sp 1 \frac1{I(r_{k+1})}\int _{r_{k+1}}^1 \dots \int
_{r_{m-1}}^1 \frac{f(r_m)}{I(r_m)} dr_m\, dr_{m-1}\, \dots dr_{k+1}
\quad \hbox{for a.e. $ x \in \Omega_I$,}
\end{equation}
and
\begin{equation}\label{E:psurt}
|\nabla ^m u(x)| =\frac {\partial\sp m u}{\partial x_n\sp
m}(x)=f(M(x_n)) \quad \hbox{for a.e. $ x \in \Omega_I$.}
\end{equation}
Thus,  if $0 \leq k \leq m$, then
\begin{equation}\label{E:relation-between-nabla-u-and-h-f}
|\nabla ^k u(x)|=\frac {\partial\sp k u}{\partial x_n\sp
k}(x)=H_I\sp{m-k}f(M(x_n)) \quad \hbox{for a.e. $x\in\Omega_I$,}
\end{equation}
where, as agreed, $H_I\sp{0}f=f$. Analogously,  if $0 \leq k \leq
m$,
\begin{equation}\label{E:relation-between-nabla-v-and-h-g}
|\nabla ^k v(x)|=\frac {\partial\sp k v}{\partial x_n\sp
k}(x)=H_I\sp{m-k}g(M(x_n)) \quad  \hbox{for a.e. $x\in\Omega_I$.}
\end{equation}
By the Leibniz Rule,
\begin{align*}
|\nabla\sp m(uv)(x)|
&=
\frac{\partial\sp m(uv)}{\partial x\sp m_n}(x)
=
\sum_{k=0}\sp m\binom{m}{k}\frac{\partial\sp ku}{\partial x\sp k_n}(x)\frac{\partial\sp {m-k}v}{\partial x\sp{m-k}_n}(x)\\
&= \sum_{k=0}\sp m\binom{m}{k}H\sp{m-k}_If(M(x_n)) H_I\sp
kg(M(x_n)) \quad  \hbox{for a.e. $x\in\Omega_I$.}
\end{align*}
Since $\Omega_I\in \mathcal J_I$,  the space $\mathcal V\sp
mX(\Omega_I)$ is a~Banach algebra by our assumption.
Therefore, in particular, there exists a~positive constant $C$ such
that
\[
\|\nabla\sp m(uv)\|_{X(\Omega_I)}\leq
C\|u\|_{\VV\sp{m}X(\Omega_I)}\|v\|_{\VV\sp{m}X(\Omega_I)},
\]
namely,
\begin{equation}\label{E:pesss}
\left\|\sum_{k=0}\sp m\binom{m}{k}H_I\sp{m-k}f(M(x_n)) H_I\sp
kg(M(x_n))\right\|_{X(\Omega_I)} \leq
C\|u\|_{\VV\sp{m}X(\Omega_I)}\|v\|_{\VV\sp{m}X(\Omega_I)}.
\end{equation}
Now, if $0 \leq k \leq m$, then
$$H_I\sp{m-k}f(M(x_n)) H_I\sp kg(M(x_n))\geq 0 \quad  \hbox{for a.e. $x\in\Omega_I$.}$$
  Therefore, we can disregard the terms with $k\neq0$
in~\eqref{E:pesss}, and obtain
\begin{equation}\label{E:extremal}
\|H_I\sp{m}f(M(x_n))\, g(M(x_n))\|_{X(\Omega_I)}\leq
C\|u\|_{\VV\sp{m}X(\Omega_I)}\|v\|_{\VV\sp{m}X(\Omega_I)}.
\end{equation}
By Lemma~\ref{R:remark}, the function $F_{m, I} : \Omega_I \to [0,
\infty)$, defined as
$$F_{m, I} (x) = H_I\sp{m}f(M(x_n)) g(M(x_n)) \qquad \hbox{for $x
\in \Omega _I$,}$$ is such that
\[
F_{m, I}\sp* =\left(H\sp{m}_If g\right)\sp*.
\]
Hence,
\begin{equation}\label{*}
\|H_I\sp{m}f(M(x_n)) g(M(x_n))\|_{X(\Omega_I)} = \|H_I\sp{m}f
g\|_{X(0,1)}.
\end{equation}
As for the terms on the right-hand side, note that,
by~\eqref{E:norm-in-vdot}, \eqref{E:psurt} and Lemma~\ref{R:remark},
\begin{equation}\label{E:hux}
\|u\|_{\VV\sp{m}X(\Omega)}=\|f\|_{X(0,1)}+\sum_{k=0}\sp{m-1}\|\nabla\sp
ku\|_{L\sp1(B)},
\end{equation}
where $B$ is any ball in $\Omega_I$. It is readily verified from
\eqref{auxdiv7} that there exists a constant $c>0$ such that
$M(x_n)\geq c$  for every $x\in B$.
Thus,
\[
\frac 1{I(r)}\leq \frac 1{I(c)}\quad \textup{if}\ x\in B\
\textup{and}\ r\in[M(x_n),1].
\]
Hence, by~\eqref{E:def-u} and~\eqref{Enec2}, if $0 \leq k\leq m-1$,
\begin{align*}
|\nabla ^k u(x)|
&\leq
\frac 1{I(c)\sp{m-k}}
\int _{M(x_n)}\sp 1
\int _{r_{k+1}}^1 \dots \int _{r_{m-1}}^1 f(r_m)
dr_m\, dr_{m-1}\, \dots dr_{k+1}
\leq
\frac {C_1}{I(c)\sp{m-k}}
\int _{M(x_n)}\sp 1f(r_m)dr_m\\
&\leq C_2\int_0\sp1f(r)\,dr \qquad \hbox{for a.e. $x\in B$},
\end{align*}
for suitable positive constants $C_1$ and $C_2$. Consequently,
by~\eqref{l1linf}, if $0 \leq k\leq m-1$
\[
\|\nabla ^k u\|_{L\sp1(B)}
\leq C_2|B|\|f\|_{L\sp1(0,1)}
\leq C_3\|f\|_{X(0,1)}
\]
 for a~suitable  constant $C_3$. Hence, via~\eqref{E:hux},
\begin{equation}\label{**}
\|u\|_{\VV\sp{m}X(\Omega)}\leq C\|f\|_{X(0,1)}
\end{equation}
for some constant $C$. Analogously,
\begin{equation}\label{***}
\|v\|_{\VV\sp{m}X(\Omega)}\leq C\|g\|_{X(0,1)}.
\end{equation}
\relax From \eqref{E:extremal}, \eqref{*}, \eqref{**} and \eqref{***} we
infer that
\begin{equation}\label{E:alto}
\|gH_I\sp{m}f\|_{X(0,1)}\leq C\|f\|_{X(0,1)}\|g\|_{X(0,1)}
\end{equation}
for some positive constant $C$.
\\
We now claim that
\begin{equation}\label{****}
\|H\sp{m}_If\|_{L\sp{\infty}(0,1)}\leq\sup_{g\geq 0,\|g\|_{X(0,1)}\leq1}\|g H\sp{m}_If\|_{X(0,1)}.
\end{equation}
Indeed, if $\lambda<\|H_I\sp{m}f\|_{L\sp{\infty}(0,1)}$, then there
exists a~set $E\subset(0,1)$ of positive measure  such that $H\sp
m_If\geq\lambda \chi _E$. Thus,
\[
\sup_{g\geq 0, \|g\|_{X(0,1)}\leq1}\|gH\sp{m}_If\|_{X(0,1)}\geq
\lambda\sup_{g\geq 0,
\|g\|_{X(0,1)}\leq1}\|g\chi_E\|_{X(0,1)}=\lambda\,,
\]
whence \eqref{****} follows. Coupling inequality \eqref{****}
with~\eqref{E:alto} tells us that
\[
\|H\sp{m}_If\|_{L\sp{\infty}(0,1)}\leq C \|f\|_{X(0,1)}
\]
for every $f \in X(0,1)$. The last inequality can be rewritten in
the form
\[
\int _{0} ^1 \frac{f(r)}{I(r)}\left(\int_0\sp
r\frac{dt}{I(t)}\right)\sp{m-1}\,dr \leq C \|f\|_{X(0,1)},
\]
and hence, by \eqref{E:assoc},
\[ \left\|\frac{1}{I(t)}\left(\int_0\sp
t\frac{ds}{I(s)}\right)\sp{m-1}\right\|_{X'(0,1)}\leq C,
\]
namely~\eqref{E:psi-in-associate}.

In particular, we have established~\eqref{E:convergence-of-I-omegaI}, and hence
also~\eqref{E:convergence-of-I-omega}. Therefore, by Remark \ref{VVW},
 the three spaces $V\sp mX(\Omega)$, $\VV\sp mX(\Omega)$ and $W\sp
mX(\Omega)$ coincide. Moreover, thanks to
\eqref{E:psi-in-associate}, we may apply
~\cite[Corollary~5.5]{CPS} and obtain thereby that $V\sp
mX(\Omega)\to L\sp{\infty}(\Omega)$ for every $\Omega\in \mathcal
J_I$. This establishes~\eqref{E:embedding-to-linfty}. Since $V\sp mX(\Omega)=\VV\sp mX(\Omega)$, the proof is
complete.
\end{proof}

The following theorem shows that the embedding into the space of
essentially bounded functions is necessary  for a Banach space to be
a~Banach algebra,  in a quite general framework.

\begin{theorem}\label{T:home}
Assume that $Z(\Omega)$ is a~Banach algebra such that
$Z(\Omega)\subset\mathcal M(\Omega)$ and
\begin{equation}\label{E:embedding-to-weak}
Z(\Omega ) \to L\sp{1,\infty}(\Omega).
\end{equation}
Then
\begin{equation}\label{E:embedding-to-l-infty}
Z(\Omega)\to L\sp{\infty}(\Omega).
\end{equation}
\end{theorem}

\begin{proof}
Since $Z(\Omega)$ is a~Banach algebra,  there exists
a~constant $C\geq 1$ such that
\begin{equation}\label{E:product-in-a}
\|uv\|_{Z(\Omega)}\leq C\|u\|_{Z(\Omega)}\|v\|_{Z(\Omega)}
\end{equation}
for every $u, v \in Z(\Omega)$. Suppose, by contradiction,
that~\eqref{E:embedding-to-l-infty} fails. Then there exists
a~function $w\in Z(\Omega)$ such that
\[
\|w\|_{L\sp{\infty}(\Omega)}>2C\|w\|_{Z(\Omega)},
\]
where $C$ is the constant from~\eqref{E:product-in-a}. In other words, the set
\[
E=\{x\in\Omega;\, |w(x)|>2C\|w\|_{Z(\Omega)}\}
\]
has positive Lebesgue measure. Fix $j\in\N$.
Applying~\eqref{E:product-in-a} $(j-1)$-times, we obtain
\[
\|w\sp j\|_{Z(\Omega)}\leq C\sp{j-1}\|w\|\sp j_{Z(\Omega)}.
\]
Combining this inequality with~\eqref{E:embedding-to-weak}  yields
\[
\lambda|\{x\in\Omega : |w(x)|^j>\lambda\}|\leq C'C\sp{j-1}\|w\|\sp
j_{Z(\Omega)}.
\]
for some constant $C'$, and for every $\lambda>0$. In particular,
the choice $\lambda=(2C\|w\|_{Z(\Omega)})\sp j$ yields
\[
(2C\|w\|_{Z(\Omega)})\sp j|E|\leq C' C\sp{j-1}\|w\|\sp
j_{Z(\Omega)},
\]
namely
\[
2\sp j|E|\leq CC\sp{-1}.
\]
However, this is impossible, since $|E|>0$ and  $j$ is arbitrary.
\end{proof}

Given a multi-index  $\gamma = (\gamma _1, \dots , \gamma _n)$,  with
$\gamma _i \in \mathbb N \cup \{0\}$ for $i=1, \dots , n$, set
$|\gamma|= \gamma _1 + \cdots + \gamma _n$, and $D^\gamma u =
\frac{\partial ^{|\gamma |}u}{\partial x_1^{\gamma _1} \dots
\partial x_n^{\gamma _n}}$ for $u : \Omega \to \mathbb R$. Moreover,
given two multi-indices $\gamma$ and $\delta$, we write $\gamma \leq
\delta$ to denote that $\gamma _i \leq \delta _i$ for $i=1, \dots ,
n$. Accordingly, by $\gamma < \delta$ we mean that $\gamma \leq \delta$ and $\gamma_i < \delta_i$ for at least one $i\in\{1,\dots,n\}$.
\par

\begin{proof}[Proof of Theorem \ref{T:algebra-restricted}]
Let $1 \leq k \leq m$. Fix any multi-index $\gamma$ satisfying
$|\gamma|\leq m-k$. Assumption~\eqref{E:power-m+k} implies that
\begin{equation}\label{E:convergent}
\int_0^1 \frac{\,ds}{I(s)}<\infty.
\end{equation}
\relax From~\eqref{E:power-m+k} and~\eqref{E:convergent}, we obtain that there exists a~positive constant $C$ such that
\begin{equation}\label{E:verifying-assumption}
\left(\int_0^t \frac{\,ds}{I(s)}\right)^{2m-|\gamma|} \leq
\left(\int_0^1 \frac{\,ds}{I(s)}\right)^{m-k-|\gamma|}
\left(\int_0^t \frac{\,ds}{I(s)}\right)^{m+k} \leq
C\varphi_{X}(t).
\end{equation}
Suppose that $\delta$ is any multi-index fulfilling $\delta \leq
\gamma$. Since~\eqref{E:power-m+k} holds, Lemma~\ref{T:thm2},
implication \textup{(vii)} $\Rightarrow$ \textup{(iii)}, combined
with~\eqref{E:optimal-Sobolev-embedding}, tells us that
\begin{equation}\label{E:embedding}
V^{m-|\delta|}X(\Omega) \to
X^{\frac{2m-|\gamma|}{m-|\gamma|+|\delta|}}(\Omega) \quad \hbox{and}
\quad V^{m-|\gamma|+|\delta|}(\Omega) \to
X^{\frac{2m-|\gamma|}{m-|\delta|}}(\Omega).
\end{equation}
On applying~\eqref{E:generalized-hoelder} with
$p=\frac{2m-|\gamma|}{m-|\gamma|+|\delta|}$ and the two embeddings
in~\eqref{E:embedding}, we deduce that
\begin{align}\label{E:product}
\|D^\delta u D^{\gamma-\delta} v\|_{X(\Omega)}
&\leq \|D^\delta u\|_{X^{\frac{2m-|\gamma|}{m-|\gamma|+|\delta|}}(\Omega)} \|D^{\gamma-\delta}v\|_{X^\frac{2m-|\gamma|}{m-|\delta|}(\Omega)}\\
&\leq C\|D^\delta u\|_{V^{m-|\delta|}X(\Omega)}
\|D^{\gamma-\delta}v\|_{V^{m-|\gamma|+|\delta|}X(\Omega)} \leq
C'\|u\|_{V^mX(\Omega)} \|v\|_{V^mX(\Omega)}, \nonumber
\end{align}
for some constants $C$ and $C'$, and for every $u, v\in
V^mX(\Omega)$. In particular, inequality \eqref{E:product} implies
that $\sum_{\delta \leq \gamma} |D^\delta u D^{\gamma-\delta} v| \in
L^1(\Omega)$. Hence, via \cite[Ex. 3.17]{AFP}, we   deduce that the
function $uv$ is $(m-k)$-times weakly differentiable and
$$
D^\gamma(uv) = \sum_{\delta \leq \gamma} \frac{\gamma!}{\delta!(\gamma-\delta)!} D^\delta u D^{\gamma-\delta}v.
$$
It follows from~\eqref{E:product} that
$$
\|D^\gamma(uv)\|_{X(\Omega)} \leq \sum_{\delta \leq \gamma}
\frac{\gamma!}{\delta!(\gamma-\delta)!} \|D^\delta u
D^{\gamma-\delta} v\|_{X(\Omega)} \leq C \|u\|_{V^mX(\Omega)}
\|v\|_{V^mX(\Omega)},
$$
for some constant $C$.
Thus,
\begin{equation}\label{dic1}
\|uv\|_{W^{m-k}X(\Omega)} = \sum_{|\gamma|\leq m-k}
\|D^\gamma(uv)\|_{X(\Omega)} \leq C \|u\|_{V^mX(\Omega)}
\|v\|_{V^mX(\Omega)}
\end{equation}
for some  constant $C$, and  for every $u, v\in V^mX(\Omega)$.
Since~\eqref{E:convergent} is in force,
$W^{m-k}X(\Omega)=\VV^{m-k}X(\Omega)$ and
$V^mX(\Omega)=\VV^{m}X(\Omega)$, up to equivalent norms. As a
consequence of \eqref{dic1}, inequality \eqref{E:P_bounded} follows.
 \par In order to prove the converse assertion, observe that, by Lemma~\ref{L:lemma}, we can assume with no
loss of generality that $I \in C^1(0, 1)$ and $I ^{n'}$ is convex.
Let  $L$, $M$, $\eta$ and $\Omega_I$ be as in
Lemma~\ref{auxdiv}. Since, by~\eqref{auxdiv1},  $\Omega_I\in
\mathcal J_{I}$, condition ~\eqref{E:P_bounded} is fulfilled. Thus,
there exists a positive constant $C$ such that
\begin{equation}\label{E:nabla-m-k}
\|\nabla^{m-k}(uv)\|_{X(\Omega_I)} \leq C \|u\|_{\VV^mX(\Omega_I)}
\|v\|_{\VV^mX(\Omega_I)}
\end{equation}
 for every $u$, $v\in
\VV^mX(\Omega_I)$. Given $f,g\in \mathcal M_+(0,1)$,  define $u,v :
\Omega_I \to [0, \infty]$ as in~\eqref{E:def-u} and~\eqref{E:def-v},
respectively.
The functions $u$ and $v$ are $m$-times weakly differentiable in
$\Omega_I$. Furthermore,
by~\eqref{E:relation-between-nabla-u-and-h-f}
and~\eqref{E:relation-between-nabla-v-and-h-g},
\begin{align}\label{lower-est}
&|\nabla^{m-k}(uv)(x)|
=\frac{\partial^{m-k}(uv)}{\partial x_n^{m-k}}(x)
=\sum_{\ell=0}^{m-k} {m-k\choose \ell} \frac{\partial^{\ell}u}{\partial x_n^{\ell}}(x)   \frac{\partial^{m-k-\ell}v}{\partial x_n^{m-k-\ell}}(x)\\
&=\sum_{\ell=0}^{m-k} {m-k\choose \ell} H_I^{m-\ell}f(M(x_n))
H_I^{k+\ell}g(M(x_n))\nonumber\\
& \geq H_I^{m}f(M(x_n))   H_I^{k}g(M(x_n)) \quad
\hbox{for a.e. $x \in \Omega_I$.}\nonumber
\end{align}
\relax From~\eqref{E:nabla-m-k}, \eqref{lower-est}, \eqref{**} and
\eqref{***}
we infer that
$$
\|H^m_If \, H^k_Ig\|_{X(0,1)} =\|H_I^{m}f(M(x_n))
H_I^{k}g(M(x_n))\|_{X(\Omega)} \leq C\|f\|_{X(0,1)} \|g\|_{X(0,1)}
$$
for some constant $C$, and for every $f, g \in \mathcal M_+(0,1)$.
Since $1\leq k\leq m\leq m+k-1$, it follows from Lemma~\ref{T:thm2},
implication \textup{(vi)} $\Rightarrow$ \textup{(vii)},
that~\eqref{E:power-m+k} holds.
\end{proof}

\begin{proof}[Proof of Theorem~\ref{C:gradients}]
{ Let $\Omega \in \mathcal J_I$,
$\|\cdot\|_{X(0,1)}$ and $1 \leq k \leq m$ be such that
~\eqref{E:power-m+k} is fulfilled. By
 Theorem~\ref{T:algebra-restricted},
there exists a~positive constant $C$ such that
\begin{equation}\label{E:remark-lower}
\sum_{l=0}\sp{m-k}\| |\nabla\sp lu| |\nabla\sp{m-k-l}v|
\|_{X(\Omega)}\leq C\|u\|_{\VV\sp mX(\Omega)}\|v\|_{\VV\sp
mX(\Omega)}
\end{equation}
for every $u,v\in\VV\sp mX(\Omega)$.}
\\
{In order to prove \eqref{E:corollary-lower}, it suffices to show
that there exists $C>0$ such that
\begin{equation}\label{dic10}
\|D\sp{\gamma}u\,D\sp{\delta}v\|_{X(\Omega)}\leq
C\|u\|_{\VV\sp{m}X(\Omega)}\|v\|_{\VV\sp{m}X(\Omega)}
\end{equation}
for every $u,v\in \VV\sp{m}X(\Omega)$ and every multi-indices
$\gamma$ and $\delta$ satisfying $|\gamma|+|\delta|=m$,
$|\gamma|\geq 1$ and $|\delta|\geq 1$. Fix such $u$, $v$, $\gamma$
and $\delta$. Without loss of generality we may assume that
$|\gamma|\leq |\delta|$. Let $\sigma$ be an arbitrary
 multi-index such that $\sigma\leq \delta$ and
 $|\sigma|=|\gamma|$. Assumption
 \eqref{E:fundamental-estimate-first} ensures that condition \eqref{E:power-m+k}
is fulfilled, with $m$ and $k$ replaced with $m-|\gamma|$ and
$|\gamma|$. Thus, owing to~\eqref{E:remark-lower}, applied  with
$u$, $v$, $m$ and $k$
 replaced by $D^\gamma u$, $D^{\delta_0} v$, $m-|\gamma|$ and $|\gamma|$,
  respectively, (and disregarding the terms with $\ell >0$ on the left-hand side) we obtain that
\begin{align*}
\|D\sp{\gamma}u \, D\sp{\delta}v\|_{X(\Omega)}
&\leq \|D\sp{\gamma}u |\nabla^{m-2|\gamma|} D^{\sigma}v|\|_{X(\Omega)}\\
&\leq C\|D\sp{\gamma}u\|_{\VV^{m-|\gamma|}X(\Omega)} \|D^{\sigma}v\|_{\VV^{m-|\gamma|}X(\Omega)}\\
&\leq C\|u\|_{\VV\sp{m}X(\Omega)}\|v\|_{\VV\sp{m}X(\Omega)},
\end{align*}
whence \eqref{dic10} follows.}
\\ {A proof of the necessity of
condition \eqref{E:fundamental-estimate-first}, under
\eqref{E:corollary-lower} and
\eqref{E:I-divided-by-power-increasing} follows along the same lines
as in the proof of \eqref{E:power-m+k} in Theorem
\ref{T:algebra-restricted}, and will be omitted for brevity.}
\end{proof}

We shall now prove a general sufficient condition for the space
$W\sp{m}X(\Omega)$  to be a~Banach algebra.

\begin{prop}\label{T:sufficiency}
Let $m,n\in\N$, $n\geq 2$, and let $I$ be a positive non-decreasing
function in $(0, 1)$. Assume that $\Omega \in \mathcal J_I$. Let
$\|\cdot\|_{X(0,1)}$ be a~rearrangement invariant~function norm.
 If ~\eqref{E:psi-in-associate}
holds, then the Sobolev space $W\sp{m}X(\Omega)$ is
a~Banach algebra.
\end{prop}

\begin{proof}
It suffices to show that, for each pair of multi-indices $\gamma$
and $\delta$
 such that $|\gamma|\leq m$ and $\delta\leq \gamma$, there exists a~positive constant $C$ such that inequality
\begin{equation}\label{E:desired}
\|D\sp{\delta}u D\sp{\gamma-\delta}v\|_{X(\Omega)} \leq
C\|u\|_{W\sp{m}X(\Omega)}\|v\|_{W\sp{m}X(\Omega)}.
\end{equation}
holds for all $u$, $v\in W\sp{m}X(\Omega)$. Indeed, such inequality implies, in particular,  that $\sum _{\delta \leq
\gamma } |D ^\delta u D ^{\gamma - \delta} v| \in L^1(\Omega )$.
Hence, once again, one can use \cite[Ex. 3.17]{AFP} to deduce that the function $uv$ is $m$-times weakly
differentiable in $\Omega$, and that, for each $\gamma$ with $1 \leq
|\gamma | \leq m$,
\begin{equation}\label{19}
D^\gamma (uv) = \sum _{\delta \leq \gamma } \frac{\gamma !}{\delta !
(\gamma -\delta)! } D^\delta u D^{\gamma-\delta} v.
\end{equation}
In order to prove~\eqref{E:desired}, let us begin by noting that,
by~\eqref{E:psi-in-associate} and~\eqref{E:dual-norm},
\begin{equation}\label{E:previous-(v)}
\int_0\sp1\frac{g(t)}{I(t)}\left(\int_0\sp
t\frac{dr}{I(r)}\right)\sp{m-1}\,dt\leq C\|g\|_{X(0,1)},
\end{equation}
 for some constant $C$ and every function $g\in\M_+(0,1)$.
Since~\eqref{E:previous-(v)} can be rewritten in the form
\[
\left\|\int_t\sp1\frac{g(s)}{I(s)}\left(\int_t\sp s\frac{dr}{I(r)}\right)\sp{m-1}\,ds\right\|_{L\sp{\infty}(0,1)}
\leq C\|g\|_{X(0,1)},
\]
it follows from~\cite[Theorem~5.1]{CPS} that $V\sp{m}X(\Omega)\hra
L\sp{\infty}(\Omega)$. Thus, by~\eqref{inclusions},  $W\sp{m}X(\Omega)\hra
L\sp{\infty}(\Omega)$ as well.
\\
Assume now that $\gamma$ is an arbitrary multi-index such that
$0\leq|\gamma|\leq m$. Then, for every $u,v\in W\sp mX(\Omega)$,
\[
\|u D\sp {\gamma}v\|_{X(\Omega)} \leq
\|u\|_{L\sp{\infty}(\Omega)}\|D\sp{\gamma}v\|_{X(\Omega)} \leq
C \|u\|_{W\sp mX(\Omega)}\|v\|_{W\sp mX(\Omega)},
\]
and, analogously,
\[
\|(D\sp {\gamma}u) v\|_{X(\Omega)}
\leq
C \|u\|_{W\sp mX(\Omega)}\|v\|_{W\sp mX(\Omega)}.
\]
This establishes~\eqref{E:desired} whenever $|\gamma|\leq m$ and
either $\delta=0$ or $\delta=\gamma$.
\par
{Assume now that $|\delta|\geq1$ and
$\delta<\gamma$. It follows from
Proposition~\ref{R:psi-implies-fundamental}
that~\eqref{E:psi-in-associate}
implies~\eqref{E:fundamental-estimate-first}, and hence also
\eqref{E:power-m+k} for every $k\in\N$. Clearly,
\begin{equation}\label{E:grad1}
\|D\sp{\delta} u \, D\sp{\gamma-\delta} v\|_{X(\Omega)} \leq
\|\,|\nabla\sp{|\delta|}u|\,|\nabla\sp{|\gamma-\delta|}v|\,\|_{X(\Omega)}.
\end{equation}
On the other hand, we claim that
\begin{equation}\label{E:grad2}
\|\,|\nabla\sp{|\delta|}u|\,|\nabla\sp{|\gamma-\delta|}v|\,\|_{X(\Omega)}
\leq C
\|u\|_{\VV\sp mX(\Omega)}\|v\|_{\VV\sp mX(\Omega)}
\end{equation}
for some positive $C$ and for every $u,v\in \VV\sp mX(\Omega)$.
Indeed, if $|\gamma|=m$, then the claim is a straightforward
consequence of Theorem~\ref{C:gradients}. If $|\gamma|<m$,  then the
claim follows from inequality \eqref{E:remark-lower}, applied with
the choice $k=m-|\gamma|$. Combining inequalities ~\eqref{E:grad1},
\eqref{E:grad2} and the first embedding in~\eqref{inclusions}
completes the proof.}
\end{proof}

\begin{proof}[Proof of Theorem~\ref{T:algebra-for-W}]
Assume that~\eqref{E:psi-in-associate} is satisfied. Then,  by
Proposition~\ref{T:sufficiency}, the space $W\sp{m}X(\Omega)$ is
a~Banach algebra for all $\Omega\in\mathcal J_{\alpha}$. Moreover,
as we have already observed,  condition~\eqref{E:psi-in-associate}
implies~\eqref{E:convergence-of-I-omega}. Therefore, by
Remark~\ref{VVW}, the spaces $W\sp{m}X(\Omega)$ and
$\VV\sp{m}X(\Omega)$ coincide. Consequently, $\VV\sp{m}X(\Omega)$ is
a~Banach algebra.
\\
The second part of the theorem is a straightforward consequence of
Proposition~\ref{T:embedding-to-linfty}.
\end{proof}

\begin{proof}[Proof of Corollary~\ref{C:embedding-to-linfty-and-algebra}]
Assume that the embedding $\VV\sp{m}X(\Omega)\to
L\sp{\infty}(\Omega)$ holds for every $\Omega\in\mathcal J_I$. Let
$\Omega_I$ be the domain defined by~\eqref{omegadiv}. Given
$f\in\M_+(0,1)$,  define $u$ by~\eqref{E:def-u}.
{It follows from~\eqref{E:relation-between-nabla-u-and-h-f} with $k=0$ that $u(x)=H\sp m_If(M(x_n))$ for
 a.e.\ $x\in\Omega_I$ and $M$ defined by~\eqref{auxdiv4}. Therefore,   $\|u\|_{L\sp{\infty}(\Omega)}=\|H\sp m_If\|_{L\sp{\infty}(0,1)}$.
 Furthermore, by~\eqref{**}, we get $\|u\|_{\VV\sp mX(\Omega)}\leq C\|f\|_{X(0,1)}$}
for some constant $C>0$.
Hence,  our
assumptions imply  that there exists a~positive constant $C$ such
that
\begin{equation}\label{5*}
\|H\sp m_If\|_{L\sp{\infty}(0,1)}\leq C\|f\|_{X(0,1)}
\end{equation}
for every nonnegative function $f\in X(0,1)$. Inequality \eqref{5*}
implies~\eqref{E:psi-in-associate}, as was again observed in the
course of proof of Proposition~\ref{T:embedding-to-linfty}. By
Theorem~\ref{T:algebra-for-W}, this tells us that
$\VV\sp{m}X(\Omega)$ is a~Banach algebra for every
$\Omega\in\mathcal J_I$.
\\
The converse implication follows at once from Proposition~\ref{T:embedding-to-linfty}.
\end{proof}

\begin{proof}[Proof of Corollary~\ref{C:linfty}]
Let $\|\cdot\|_{X(0,1)}$ be a~rearrangement-invariant~function norm.
Then, by the assumption and~\eqref{l1linf},
condition~\eqref{E:psi-in-associate} is satisfied. Hence, owing to
Theorem~\ref{T:algebra-for-W}, the space $\VV\sp{m}X(\Omega)$ is
a~Banach algebra.
\end{proof}

\begin{proof}[Proof of Theorem~\ref{T:john}]
Assume first that $\VV\sp{m}X(\Omega)$
is a~Banach algebra. Due to~\eqref{Ijohn},
condition~\eqref{E:convergence-of-I-omega} is satisfied. Hence, by
Remark~\ref{VVW}, the spaces $\VV\sp{m}X(\Omega)$ and
$W\sp{m}X(\Omega)$ coincide. In particular, $W\sp{m}X(\Omega)$ is
a~Banach algebra. Furthermore, by~\eqref{l1linf} and trivial inclusions, we clearly have
\[
W\sp{m}X(\Omega)\to X(\Omega) \to L\sp1(\Omega)\to L\sp{1,\infty}(\Omega),
\]
hence the assumption~\eqref{E:embedding-to-weak} of Theorem~\ref{T:home} is satisfied. Thus, as a~special case of Theorem~\ref{T:home} we
obtain that $W\sp{m}X(\Omega)\to L\sp{\infty}(\Omega)$.
Therefore,  $W\sp{m}X(\Omega)\to L\sp{\infty}(\Omega)$. On the other
hand, by~\cite[Theorem~6.1]{CPS}, this embedding is equivalent to
the inequality
 \[
\int_0\sp1g(s)s\sp{\frac mn-1}\,ds\leq C \|g\|_{X(0,1)}
\]
 for some constant $C$, and for every nonnegative function $g\in X(0,1)$. Hence, via
 the very definition of associate function norm, we
obtain~\eqref{E:john-condition}.

Conversely, assume that~\eqref{E:john-condition} holds. Since
$\Omega$ is a~John domain, inequality~\eqref{isop-ineq} is satisfied
with $I(t)=t\sp{\frac1{n'}}$, $t\in(0,1)$. Hence, it follows from
Theorem~\ref{T:algebra-for-W}  that the space $\VV\sp{m}X(\Omega)$
is a~Banach algebra.
\end{proof}

\begin{proof}[Proof of Proposition~\ref{T:lz}]
If $I(t)=t\sp{\alpha}$, with $\alpha\in[\frac1{n'},\infty)$, then
condition~\eqref{E:I-divided-by-power-increasing} is clearly
satisfied. Hence, by Theorem~\ref{T:algebra-for-W},  the space
$\VV\sp mL\sp{p,q;\beta}(\Omega)$ is a~Banach algebra
for every $\Omega\in\mathcal J_\alpha$ if and only
if~\eqref{E:psi-in-associate} holds. Owing to Remark
\ref{conv-int-I},  condition \eqref{E:psi-in-associate} entails that
$\alpha<1$. Thus,
\[
\frac1{I(t)}\left(\int_0\sp t\frac{ds}{I(s)}\right)\sp{m-1}=(\tfrac
{1}{1-\alpha})\sp{m-1} t\sp{m(1-\alpha)-1}\quad \hbox{for
$t\in(0,1)$.}
\]
Therefore, it only remains to analyze under which conditions the
power function $t\sp{m(1-\alpha)-1}$ belongs to
$(L\sp{p,q;\beta})'(0,1)$. It is easily  verified, via
\eqref{E:lz_assoc}, that this is the case if and only if one of the
conditions in~\eqref{E:lz-algebra} is satisfied.
\end{proof}

\begin{proof}[Proof of Proposition~\ref{T:orlicz}]
By the same argument as in the proof of Proposition~\ref{T:lz} we deduce that for $\VV\sp mL\sp{A}(\Omega)$ to be a~Banach algebra it is necessary that $\alpha<1$.
 If $m \geq \tfrac 1{1-\alpha}$, then \eqref{linfty1}
holds with $I(t)=t\sp{\alpha}$, $t\in (0,1)$, hence, by Corollary
\ref{C:linfty}, $\VV ^m L^A(\Omega)$ is a Banach algebra whatever
$A$ is. If $m < \tfrac 1{1-\alpha}$, then assumption
\eqref{E:psi-in-associate} of Theorem \ref{T:algebra-for-W} is
equivalent to
\begin{equation}\label{orliczproof1}
\|r^{m(1-\alpha )-1}\|_{L^{\widetilde A}(0, 1)} < \infty.
\end{equation}
The latter condition is, in turn, equivalent to
\begin{equation}\label{orliczproof2}
\int _0 \widetilde A \big(r^{m(1-\alpha )-1}\big)\, dr < \infty\,,
\end{equation}
namely to
\begin{equation}\label{orliczproof3}
\int ^\infty \frac{\widetilde A (t)}{t^{1+ \frac{1}{1-
m(1-\alpha)}}}\, dt < \infty.
\end{equation}
By \cite[Lemma 2.3]{cianchioptimal}, equation \eqref{orliczproof3}
is equivalent to
$$\int ^\infty
\bigg(\frac t{A(t)}\bigg)^{\frac {(1-\alpha)m}{1- (1-\alpha )m}}dt <
\infty.$$
\end{proof}

\begin{proof}[Proof of Proposition~\ref{T:lz-reduced}]
Condition  ~\eqref{E:I-divided-by-power-increasing} is satisfied if
$I(t)=t\sp{\alpha}$, with $\alpha\in[\frac1{n'},\infty)$. By
Theorem~\ref{T:algebra-restricted}, condition~\eqref{E:P_bounded}
for $X(\Omega)=L\sp{p,q;\beta}(\Omega)$ holds for every
$\Omega\in\mathcal J_\alpha$ if and only if~\eqref{E:power-m+k} is
in force. In turn, inequality~\eqref{E:power-m+k}
entails~\eqref{E:convergence-of-I-omegaI}, whence $\alpha<1$. Now,
\begin{equation}\label{E:lz_fundamental}
\varphi_{L\sp{p,q;\beta}}(t) \approx
\begin{cases}
t\sp{\frac 1p}(\log\frac 2t)\sp{\beta}&\textup{if}\ 1<p<\infty,\ 1\leq q\leq\infty,\ \beta\in\R;\\
t(\log\frac 2t)\sp{\beta}&\textup{if}\ p=1,\ q=1,\ \beta\geq0;\\
1&\textup{if}\ p=\infty,\ q=\infty,\ \beta=0;\\
(\log\frac 2t)\sp{\beta+\frac 1q}&\textup{if}\ p=\infty,\ 1\leq q<\infty,\ \beta+\frac 1q<0,
\end{cases}
\end{equation}
for $t\in(0,1)$ (see e.g. \cite[Lemma~9.4.1, page~318]{fs-1}).
On the other hand, \begin{equation}\label{int-alpha}\bigg(\int _0^t
s^{-\alpha}\, ds\bigg)^{m+k} = (1-\alpha)^{-m-k}
t\sp{(1-\alpha)(m+k)}\quad \hbox{for $t \in (0,1)$.}
\end{equation}
Thus, inequality~\eqref{E:power-m+k} holds, with
$X(0,1)=L\sp{p,q;\beta}(0,1)$,
 if and only if
one of the conditions in~\eqref{E:lz-reduced} is satisfied.
\end{proof}

\begin{proof}[Proof of Proposition~\ref{T:orlicz-reduced}] As observed in the above proof,
we may assume that $\alpha <1$, and, by Theorem
\ref{T:algebra-restricted}, reduce  \eqref{E:P_bounded}, with
$I(t)=t\sp{\alpha}$, $t\in (0,1)$, and $X(\Omega)=L^A(\Omega)$, to
the validity of \eqref{E:power-m+k}. It is easily seen  that
$$\varphi _{L^A} (t) = \frac 1{A^{-1}(1/t)} \quad \hbox{for $t \in
(0,1)$.}$$
 Thus, the conclusion follows via \eqref{int-alpha}.
\end{proof}

\end{document}